\documentclass[reqno, 10pt]{article}

\usepackage{fullpage}

\usepackage[centertags]{amsmath}
\usepackage{amsmath}
\usepackage{amsfonts}
\usepackage{amssymb}
\usepackage{amsthm}
\usepackage{enumerate}
\usepackage[ruled,linesnumbered,vlined]{algorithm2e}
\usepackage{subfig}

\usepackage{newlfont}
\usepackage{color}

\usepackage{authblk}

\usepackage{stmaryrd}
\SetSymbolFont{stmry}{bold}{U}{stmry}{m}{n}

\usepackage{bbm}
\include{amsthm_sc}

\usepackage{stmaryrd}
\usepackage{mathrsfs}

\usepackage{fancyhdr}
\usepackage{graphicx}
\usepackage{fancybox}
\usepackage{setspace}
\usepackage{cleveref}

\newtheorem{thm}{Theorem}

\newtheorem{lem}[thm]{Lemma}

\theoremstyle{definition}
\newtheorem{defn}[thm]{Definition}
\newtheorem{assump}[thm]{Assumption}

\theoremstyle{remark}
\newtheorem{rem}[thm]{Remark}

\newcommand{\R} {\mathbb{R}}
\newcommand{\C} {\mathbb{C}}

\newcommand{\E} {\mathbb{E}}

\newcommand{\p} {\mathbb{P}}
\renewcommand{\S} {\mathbb{S}}

\DeclareMathOperator{\Tr}{Tr}

\DeclareMathOperator{\err}{\mathrm{err}}
\DeclareMathOperator{\erfc}{\mathrm{erfc}}

\DeclareMathOperator{\im}{\mathrm{Im}}

\newcommand{\caP}{{\mathcal P}}
\newcommand{\caN}{{\mathcal N}}
\newcommand{\caO}{{\mathcal O}}
\newcommand{\caX}{{\mathcal X}}

\newcommand{\bse}{{\boldsymbol e}}
\newcommand{\bsH}{{\boldsymbol H}}
\newcommand{\bsv}{{\boldsymbol v}}
\newcommand{\bsw}{{\boldsymbol w}}
\newcommand{\bsx}{{\boldsymbol x}}
\newcommand{\bsy}{{\boldsymbol y}}

\newcommand{\wt}{\widetilde}

\newcommand{\beq}{ \begin{equation} }
\newcommand{\eeq}{ \end{equation} }

\newcommand{\dd}{\mathrm{d}}
\newcommand{\ii}{\mathrm{i}}

\newcommand{\tM}{\widetilde{M}}

\newcommand{\fh}{F^H}
\newcommand{\gh}{G^H}
\newcommand{\SNR}{\omega}

\makeatletter
\def\blfootnote{\xdef\@thefnmark{}\@footnotetext}
\makeatother

\numberwithin{equation}{section} 
\numberwithin{thm}{section}

\title{Weak detection in the spiked Wigner model}

\author{Hye Won Chung \footnote{School of Electrical Engineering, KAIST, Daejeon, 34141, Korea
		\newline email: \texttt{hwchung@kaist.ac.kr}}
	and Ji Oon Lee\footnote{Department of Mathematical Sciences, KAIST, Daejeon, 34141, Korea
		\newline email: \texttt{jioon.lee@kaist.edu}}}

\date{\today}

\begin{document}

\maketitle

\begin{abstract}
We consider the weak detection problem in a rank-one spiked Wigner data matrix where the signal-to-noise ratio is small so that reliable detection is impossible. We propose a hypothesis test on the presence of the signal by utilizing the linear spectral statistics of the data matrix. The test is data-driven and does not require prior knowledge about the distribution of the signal or the noise. When the noise is Gaussian, the proposed test is optimal in the sense that its error matches that of the likelihood ratio test, which minimizes the sum of the Type-I and Type-II errors. If the density of the noise is known and non-Gaussian, the error of the test can be lowered by applying an entrywise transformation to the data matrix. We establish a central limit theorem for the linear spectral statistics of general rank-one spiked Wigner matrices as an intermediate step.
\end{abstract}

\blfootnote{This paper was presented in part at the 36th International Conference on Machine Learning 2019 \cite{chung2019weak}.}

\section{Introduction} \label{sec:intro}

It is of fundamental importance in statistics to detect signals from given noisy data. If the data is given as a square matrix, it is common to apply principal component analysis (PCA), which is to analyze the data by the largest eigenvalue and the eigenvector corresponding to it. For a null model where no signal is present, the data is pure noise in which case the behavior of the largest eigenvalue can be precisely predicted by random matrix theory \cite{TW1994,TW1996,Johnstone2001,Tao-Vu2010,EYY2012}. If the signal is present, the data matrix is of `signal-plus-noise' type, and it corresponds to a `deformed random matrix.' 
In the simple yet realistic case where the signal is in the form of a vector, the model is often referred to as a `spiked random matrix.'

When the signal is an $N$-dimensional vector and the data is an $N \times N$ real symmetric matrix, one of the most natural signal-plus-noise models is of the form
\beq \label{eq:model}
	M = \sqrt{\lambda} \bsx \bsx^T + H,
\eeq
where the signal $\bsx \in \R^N$ and the noise $H$ is an $N \times N$ Wigner matrix. (See Definitions \ref{defn:Wigner} and \ref{defn:spiked_Wigner}.) The parameter $\lambda$ corresponds to the signal-to-noise ratio (SNR).
The spiked Wigner model is widely used as a low-rank model for which PCA can be applied to detect or recover signal from noisy high dimensional data. It can be used in the signal detection/recovery problems such as community detection \cite{Abbe2017}, submatrix localization \cite{Butucea2013}, and angular synchronization \cite{singer2011angular,bandeira2013cheeger}. Community detection is a central problem in data science where the goal is to identify groups of nodes, called communities, with stronger interactions than the rest of the graph. In the community detection, the spike $\bsx\in\{1,-1\}^N$ indicates communities each node belongs to and the data matrix $M$ models noisy pairwise interactions with different means depending on whether the corresponding nodes are in the same community or not. In the submatrix localization, the task is to detect within a large Gaussian matrix small blocks with atypical mean. In these two problems, the main goals are to recover $\bsx$ from the noisy data matrix $M$ or to test the presence of community structures or submatrices with atypical mean in the data when one cannot reliably detect $\bsx$. 
Another example is the angular synchronization problem that aims to recover unknown angles $\bsx=(\exp(i\theta_1),...,\exp(i\theta_N))$ from pairwise noisy measurements of $\theta_i-\theta_j$ mod $2\pi$. 

For the Wigner matrix $H$, which represents the noise, we assume that $H_{ij}$ are independent random variables with mean $0$ and variance $N^{-1}$. With the assumption, the spectral norm of $H$, $\| H \| \to 2$ as $N \to \infty$ almost surely. Thus, when the signal is normalized so that $\| \bsx \|_2 = 1$, the strengths of the signal $\| \bsx \bsx^T \|$ and the noise $\| H\|$ are comparable. If the parameter $\lambda$, which corresponds to the signal-to-noise ratio (SNR), is very large $(\lambda \gg 1)$ or small $(\lambda \ll 1)$, a perturbation argument can be applied for PCA; if $\lambda \gg 1$, the difference between the largest eigenvalues of $M$ and $\sqrt{\lambda} \bsx \bsx^T$ is negligible, and if $\lambda \ll 1$, the largest eigenvalue of $M$ cannot be distinguished from that of $H$.

The case $\lambda \sim 1$ in \eqref{eq:model} has been intensively studied. The first and the most notable result in this direction was obtained by Baik, Ben Arous, and P\'ech\'e \cite{BBP2005} for complex Wishart matrices, which is of the form $X^* X$ where $X$ is a rectangular matrix with independent complex Gaussian entries, and later extended to more general sample covariance matrices \cite{Paul2007,Nadler2008,JohnstoneLu2009}. Similar results were proved for (both real and complex) Wigner matrices under various assumptions \cite{Peche2006,FeralPeche2007,CapitaineDonatiFeral2009,Raj2011}. In these results, the largest eigenvalue exhibits the following phase transition; when $\lambda > 1$, the largest eigenvalue separates from the other eigenvalues of $M$ and converges to $\sqrt{\lambda} + \frac{1}{\sqrt{\lambda}}$, which is strictly larger than $2$, whereas for $\lambda < 1$, the behavior of the largest eigenvalue coincides with that of the pure noise model. In the former case, the eigenvector corresponding to the largest eigenvalue has nontrivial correlation with the signal $\bsx$ and the signal can be detected and recovered by PCA. We refer to the work of Benaych-Georges and Nadakuditi \cite{Raj2011} for more detail on the behavior of the largest eigenvalues and corresponding eigenvectors, including their phase transition.

When $\lambda < 1$, contrary to the case $\lambda > 1$, simple application of PCA does not provide the information on the signal. In this case, the spectral norm of $M$ converges to $2$ and the behavior of the largest eigenvalue cannot be distinguished from that of the null model $H$. It is then natural to ask whether the presence of the signal is detectable, and if so, which tests allow us to detect it in the regime $\lambda < 1$.

The question on the detectability was considered in the seminal work by Montanari, Reichman, and Zeitouni \cite{Montanari2017}, where it was proved that no tests based on the eigenvalues can reliably detect the signal if the noise $H$ is a random matrix from the Gaussian Orthogonal Ensemble (GOE). For a non-Gaussian Wigner matrix $H$, Perry, Wein, Bandeira, and Moitra \cite{Perry2018} assumed that the signal $\bsx$ is drawn from a distribution $\caX$, which they called the spike prior, and found the critical value for $\lambda \leq 1$ in terms of $\caX$ and $H$ below which no tests based on the eigenvalues can reliably detect signal. Moreover, they also established a transformation acting on each entry of the data matrix after which the signal can be detected via the largest eigenvalue whenever $\lambda$ is larger than the critical value that is strictly less than $1$ if the noise is non-Gaussian.

For the subcritical case, El Alaoui, Krzakala, and Jordan \cite{AlaouiJordan2018} studied the weak detection, i.e., a test with accuracy better than a random guess. More precisely, they considered the hypothesis testing problem between the null hypothesis that $\lambda = 0$ and the alternative hypothesis that $M$ is generated with a fixed $\lambda > 0$. Assuming that the entries of $\sqrt{N} \bsx$ are i.i.d. random variables with bounded support and the noise $H$ is Gaussian, it was proved that the error from the likelihood ratio (LR) test, which is the optimal test in minimizing the sum of the Type-I and the Type-II error, converges to
\beq \label{eq:optimal_error}
	\erfc \left( \frac{1}{4} \sqrt{-\log (1-\lambda) - \lambda} \right)
\eeq
if the variance of diagonal entries $H_{ii}$ tends to infinity. We remark that Onatski, Moreira, and Hallin \cite{OnatskiMoreiraHallin} considered the weak detection for real Wishart matrices and obtained a Gaussian limit of the log LR. 

While the likelihood ratio test is optimal as predicted by the Neyman--Pearson lemma, LR tests require substantial knowledge of the distribution of $x$, called prior, which may not be known a priori in many practical cases such as community detection problems. Thus, it is desirable to design a test that does not require any knowledge on the signal. For community detection problem in the stochastic block model, Banerjee and Ma \cite{BanerjeeMa2017} proposed a test based on the linear spectral statistics (LSS). More precisely, denoting by $\mu_1, \dots, \mu_N$ the eigenvalues of the data matrix, they considered the LSS
\beq \label{eq:LSS}
	L_N(f) = \sum_{i=1}^N f(\mu_i)
\eeq
with $f(x) = x^k$ for positive integers $k$ and showed that asymptotically optimal error is achieved by a linear combination of the LSS.

The results in \cite{AlaouiJordan2018,BanerjeeMa2017} shed lights on the weak detection problem. However, the analysis in these results seems to be restricted to the specific distributions of the noise - the Gaussian distribution in \cite{AlaouiJordan2018} and the Bernoulli distribution in \cite{BanerjeeMa2017}. Moreover, the signal considered in the previous works is completely delocalized, i.e., $\| \bsx \|_{\infty} = O(1/\sqrt{N})$, which may lose its validity if the signal is sparse.

In this paper, we construct an optimal and universal test that detects the absence or presence of signal in \eqref{eq:model} based on LSS for any $\bsx$ with $\| \bsx \|_2=1$ and for any Wigner matrix $H$. 
Our proposed test has the following properties:

\begin{itemize}

\item Universality 1: For any deterministic or random $\bsx$, the proposed test and its error do not change, and thus we do not need any prior information on $\bsx$. Note that the LR test requires the prior information on $\bsx$.

\item Universality 2: The proposed test and its error depend on the distribution of the noise $H$ only through the variance of the diagonal entries and the fourth moment of the off-diagonal entries. The entries $H_{ij}$ do not need to be identically distributed but just independent.

\item Optimality 1: The proposed test is with the lowest error among all tests based on LSS.

\item Optimality 2: When the noise is Gaussian, the error of the proposed test with low computational complexity converges to the optimal limit \eqref{eq:optimal_error} obtained in \cite{AlaouiJordan2018}.

\item Data-driven test: The various quantities in the proposed test can be estimated from the observed data.

\end{itemize}

For the case where the noise is non-Gaussian, the error from the LR test is not known and hence it is not clear whether the proposed test achieves the optimal error. It seems natural to adapt the entrywise transformation in \cite{Perry2018} to further improve the proposed test, since the results in \cite{Perry2018} show that the PCA performs better with the transformation. When we apply it to our model for the hypothesis testing, we may consider the following two scenarios with respect to the transformation: (1) the transformation effectively changes the SNR and thus the error from the proposed test decreases after the transformation, or (2) the transformation increases the largest eigenvalue but also changes other eigenvalues in a way that the proposed test, which uses all eigenvalues contrary to the PCA that only uses the largest eigenvalue, does not improve (or even performs worse).

In this paper, we prove that the first scenario is true, i.e., the transformation indeed effectively changes the SNR from $\lambda$ to $\lambda \fh$, where $\fh$ is the Fisher information of the density function of the (normalized) off-diagonal entries in $H$. Since $\fh \geq 1$, and the strict inequality holds if the noise is non-Gaussian, the error from our test decreases in general after the transformation. In summary, our proposed test also has the following property:

\begin{itemize}

\item Entrywise transformation: If the density function of the noise matrix is known, which is non-Gaussian, the test can be further improved by an entrywise transformation that effectively increases the SNR.

\end{itemize}

As a generalization of the model, we briefly introduce the model where the value of the SNR for the alternative hypothesis is not known and explain how to construct in this case an adaptive test that performs better than a random guess if the prior distribution of SNR under the alternative hypothesis is known. The idea of the test can be further extended to more general models such as rank-$k$ spiked Wigner matrices with $k>1$ or spiked sample covariance matrices. Such generalizations will be considered in our future papers.

The main mathematical achievement of the present paper is the central limit theorem (CLT) for the LSS of arbitrary analytic functions for the random matrix in \eqref{eq:model}. The fluctuation of the LSS is not only of fundamental importance per se in random matrix theory, but also applicable to various applications such as the fluctuations of the free energy of the spherical spin glass \cite{Baik-Lee2016,Baik-Lee2017}. The LR in the weak detection problem with Gaussian noise is directly related to the free energy of spin glass as in \cite{AlaouiJordan2018}. To our best knowledge, however, the CLT for spiked Wigner matrices was proved only for the case where the signal $\bsx = \boldsymbol{1} := \frac{1}{\sqrt N}(1, 1, \dots, 1)^T$ \cite{Baik-Lee2017}.

The rest of the paper is organized as follows. In Section \ref{sec:prelim}, we define the model and introduce previous results. In Section \ref{sec:main}, we state the main result and describe the algorithm for the proposed test. In Section \ref{sec:adaptive}, we explain an adaptive test that does not require the precise information on SNR. In Section \ref{sec:trans}, we apply the entrywise transformation and state the results for the improved test. In Section \ref{sec:simul}, we conduct numerical simulations with the algorithm for the proposed test and compare the outcomes with the theoretical results. General results on the CLT for the LSS are collected in Section \ref{sec:CLT}. We conclude the paper in Section \ref{sec:conclusion} with the summary of our results and future research directions. Proof of the theorems and technical results can be found in Appendices.

\subsubsection*{Notational remarks}

We use the standard big-O and little-o notation: $a_N = O(b_N)$ implies that there exists $N_0$ such that $a_N \leq C b_N$ for some constant $C>0$ independent of $N$ for all $N \geq N_0$; $a_N = o(b_N)$ implies that for any positive constant $\epsilon$ there exists $N_0$ such that $a_N \leq \epsilon b_N$ for all $N \geq N_0$.

For an event $\Omega$, we say that $\Omega$ holds with high probability if for any (large) $D > 0$ there exists $N_0 \equiv N_0 (D)$ such that $\p(\Omega^c) < N^{-D}$ whenever $N > N_0$.

For $X$ and $Y$, which can be deterministic numbers and/or random variables depending on $N$, we use the notation $X = \caO(Y)$ if for any (small) $\epsilon > 0$ and (large) $D > 0$ there exists $N_0 \equiv N_0 (\epsilon, D)$ such that $\p(|X|>N^{\epsilon} |Y|) < N^{-D}$ whenever $N > N_0$. Roughly speaking, $X = \caO(Y)$ means that with high probability $|X|$ is not significantly larger than $|Y|$.

\section{Preliminaries} \label{sec:prelim}

We begin by defining the matrix in \eqref{eq:model} more precisely. 
The Wigner matrix is defined as follows:

\begin{defn}[Wigner matrix] \label{defn:Wigner}
We say an $N \times N$ random matrix $H = (H_{ij})$ is a (real) Wigner matrix if $H$ is a symmetric matrix and $H_{ij}$ ($1\leq i \leq j\leq N$) are independent real random variables satisfying the following conditions:
\begin{itemize}
\item For all $i\leq j$, all moments of $H_{ij}$ are finite and $\E[H_{ij}]=0$.
\item For all $i<j$, $N \E[H_{ij}^2]=1$, $N^{\frac{3}{2}} \E[H_{ij}^3]=w_3$, and $N^2 \E[H_{ij}^4]=w_4$ for some constants $w_3\in \R$ and $w_4> 0$. 
\item For all $i$, $N\E[H_{ii}^2]=w_2$ for a constant $w_2\geq 0$. 
\end{itemize}
\end{defn}

Note that we do not assume $H_{ij}$ are identically distributed. Our results in the paper hold as long as they are independent and the first four moments of off-diagonal entries (and the first two moments of diagonal entries) match. We also remark that the finite moment condition for $H_{ij}$ can be relaxed further; it is believed that the finiteness of the fourth moment suffices. However, for the sake of brevity, we do not pursue it in the current paper.

The signal-plus-noise model we consider is a (rank-one) spiked Wigner matrix, which is defined as follows:
	
\begin{defn}[Spiked Wigner matrix] \label{defn:spiked_Wigner}
We say an $N \times N$ random matrix $M = \sqrt{\lambda} \bsx \bsx^T + H$ is a spiked Wigner matrix with a spike $\bsx$ and signal-to-noise ratio (SNR) $\lambda$ if $\bsx = (x_1, x_2, \dots, x_N) \in \R^N$ with $\| \bsx \|_2 = 1$ and $H$ is a Wigner matrix.
\end{defn}

Denote by $\p_1$ the joint probability of the observation, a spiked Wigner matrix, with $\lambda  = \SNR > 0$ and $\p_0$ with $\lambda = 0$. If $H$ is a GOE matrix, where $H_{ij}$ are Gaussian with $N\E[H_{ii}^2]=2$, and $\bsx$ is drawn from the spike prior $\caX$, the likelihood ratio is given by
\[
	\frac{\dd \p_1}{\dd \p_0} = \int \exp \Big( \frac{N}{2} \sum_{i, j=1}^N \big( \sqrt{\SNR} M_{ij} x_i x_j - \frac{\SNR}{2} x_i^2 x_j^2 \big) \Big) \dd \caX(\bsx).
\]
For the spherical prior, i.e., $\caX$ is the uniform distribution on the unit sphere, with the spike $\bsx = \boldsymbol{1}$, it was proved in \cite{Baik-Lee2016,Baik-Lee2017} that
\[
	\log \frac{\dd \p_1}{\dd \p_0} \Rightarrow \caN \left( \pm \frac{1}{4} \log \left( \frac{1}{1-\SNR} \right), \frac{1}{4} \log \left( \frac{1}{1-\SNR} \right) \right),
\]
where the plus sign holds under the alternative $M \sim \p_1$ and the minus sign holds under the null $M \sim \p_0$. (See Section 3.1 of \cite{Baik-Lee2016} and Theorem 1.4 of \cite{Baik-Lee2017} with $\beta = \sqrt{\SNR}/2$.) For the i.i.d. bounded prior, i.e., the entries of $\sqrt{N} \bsx$ are i.i.d. random variables with bounded support, the same result was proved in \cite{AlaouiJordan2018}.

The proof of the convergence of $\frac{\dd \p_1}{\dd \p_0}$ in \cite{Baik-Lee2016,Baik-Lee2017} is based on the recent development of random matrix theory, especially the study of the linear spectral statistics. For a Wigner matrix $H$, if we let $\lambda_1 \geq \lambda_2 \geq \dots \geq \lambda_N$ be the eigenvalues of $H$, then for any continuous function $f$ defined on a neighborhood of $[-2, 2]$,
\[
	\frac{1}{N} \sum_{i=1}^N f(\lambda_i) \to \int_{-2}^2 \frac{\sqrt{4-x^2}}{2\pi} f(x) \, \dd x 
\]
almost surely, which is the celebrated Wigner semicircle law. The fluctuation of $\frac{1}{N} \sum_i f(\lambda_i)$ about its limit is a subject of intensive study in random matrix theory, and it is natural to introduce the linear spectral statistic (LSS) defined in \eqref{eq:LSS} for the analysis. The CLT for the LSS asserts that
\beq \begin{split} \label{eq:CLT_LSS}
	\left( L_N(f) - N \int_{-2}^2 \frac{\sqrt{4-x^2}}{2\pi} f(x) \, \dd x \right) 
	\Rightarrow \caN(m_H(f), V_H(f)),
\end{split} \eeq
where the right-hand side denotes a Gaussian random variable with mean $m_H(f)$ and variance $V_H(f)$. Note that the size of the fluctuation is of $N^{-1}$ and much smaller than that of the conventional central limit theorem, $N^{-\frac{1}{2}}$.

For the spiked Wigner matrices, the CLT for the LSS has been proved only for the case $\bsx = \boldsymbol{1} := \frac{1}{\sqrt N}(1, 1, \dots, 1)^T$ in \cite{Baik-Lee2017}. Let $\mu_1 \geq \mu_2 \geq \dots \geq \mu_N$ be the eigenvalues of a spiked Wigner matrix with a spike $\bsx$ and SNR $\lambda$. If $\bsx = \boldsymbol{1}$, then
\beq \begin{split} \label{eq:CLT}
	\left( \sum_{i=1}^N f(\mu_i) - N \int_{-2}^2 \frac{\sqrt{4-x^2}}{2\pi} f(x) \, \dd x \right) 
	\Rightarrow \caN(m_M(f), V_M(f)),
\end{split} \eeq
A remarkable fact in \eqref{eq:CLT} is that the variance $V_M(f)$ is equal to $V_H(f)$, the variance from the Wigner case, whereas the mean $m_M(f)$ is different from $m_H(f)$ unless $\lambda=0$. (See Theorem \ref{thm:CLT} in Section \ref{sec:CLT} for the precise formulas for $m_M(f)$ and $V_M(f)$.) It turns out that the same CLT holds for any spike $\bsx$ as in Theorem \ref{thm:CLT}, and the LSS provides us a test statistic for a hypothesis testing.

\section{Main Results} \label{sec:main}

Let us denote by $\bsH_0$ the null hypothesis and $\bsH_1$ the alternative hypothesis, i.e.,
\[
	\bsH_0 : \lambda = 0, \qquad \bsH_1 : \lambda = \SNR > 0.
\]
Suppose that the value $\SNR$ for $\bsH_1$ is known and our task is to detect whether the signal is present from a given data matrix $M$. If we construct a test based on the LSS for the hypothesis testing, it is obvious that we need to maximize
\beq \label{eq:target}
	\left| \frac{m_M(f) - m_H(f)}{\sqrt{V_M(f)}} \right|.
\eeq
In Theorem \ref{thm:optimize} in Section \ref{sec:CLT}, we prove that the maximum of \eqref{eq:target} is attained if and only if $f(x) = C_1 \phi_{\SNR}(x) + C_2$ for some constants $C_1$ and $C_2$, where
\beq \begin{split} \label{eq:phi}
	\phi_{\SNR}(x) := \log \left( \frac{1}{1-\sqrt{\SNR} x + \SNR} \right) + \sqrt{\SNR} \left( \frac{2}{w_2} - 1 \right) x + \SNR \left( \frac{1}{w_4-1} - \frac{1}{2} \right) x^2.
\end{split} \eeq
Thus, it is natural to define the test statistic $L_{\SNR}$ by
\beq \begin{split} \label{eq:L_lambda}
	L_{\SNR} &:= \sum_{i=1}^N \phi_{\SNR}(\mu_i) -  N \int_{-2}^2 \frac{\sqrt{4-z^2}}{2\pi} \phi_{\SNR}(z) \, \dd z \\
	&= - \log \det \left( (1+\SNR)I - \sqrt{\SNR} M \right) + \frac{\SNR N}{2} 
	+ \sqrt{\SNR} \left( \frac{2}{w_2} - 1 \right) \Tr M 
	+ \SNR \left( \frac{1}{w_4-1} - \frac{1}{2} \right) (\Tr M^2 - N).
\end{split} \eeq
Under $\bsH_0$, it is direct to see from Theorem 1.1 of \cite{Bai-Yao2005} and Section 3.1 of \cite{Baik-Lee2016} that
\[
	L_{\SNR} \Rightarrow \caN(m_0, V_0),
\]
where
\beq \begin{split} \label{eq:m_0}
	m_0 = -\frac{1}{2} \log(1-\SNR) 
	+ \left(\frac{w_2 -1}{w_4-1} -\frac{1}{2} \right) \SNR + \frac{(w_4 -3) \SNR^2}{4}
\end{split} \eeq
and
\beq \begin{split} \label{eq:V_H}
	V_0 = -2 \log(1-\SNR) 
	+ \left( \frac{4}{w_2} -2 \right) \SNR + \left( \frac{2}{w_4-1} - 1 \right) \SNR^2.
\end{split} \eeq
Our first main result is the CLT for $L_{\SNR}$
under $\bsH_1$.

\begin{thm} \label{thm:main}
Let $M$ be a spiked Wigner matrix in Definition \ref{defn:spiked_Wigner} with $\lambda = \SNR \in (0, 1)$. Suppose $w_2>0$ and $w_4>1$. Then, for any spike $\bsx$ with $\| \bsx \|_2 = 1$,
\beq
	L_{\SNR} \Rightarrow \caN(m_+, V_0)\,.
\eeq
The mean of the limiting Gaussian distribution is given by
\beq \begin{split} \label{eq:m_+}
	m_+ = m_0 - \log(1-\SNR) 
	+ \left( \frac{2}{w_2} - 1 \right) \SNR + \left( \frac{1}{w_4-1} - \frac{1}{2} \right) \SNR^2
\end{split} \eeq
and the variance $V_0$ is as in \eqref{eq:V_H}.
\end{thm}

Theorem \ref{thm:main} is a direct consequence of a general CLT in Theorem \ref{thm:CLT} in Section \ref{sec:CLT}. See Supplementary Material for more details.

In the simplest case with $w_2=2$ and $w_4=3$, e.g., when $H$ is a GOE matrix,
\beq \label{eq:test_0}
	L_{\SNR} \Rightarrow \caN(-\frac{1}{2} \log(1-\SNR), -2 \log(1-\SNR))
\eeq
under $\bsH_0$ and
\beq \label{eq:test_1}
	L_{\SNR} \Rightarrow \caN(-\frac{3}{2} \log(1-\SNR), -2 \log(1-\SNR))
\eeq
under $\bsH_1$ as shown in Figure \ref{fig:normal}.

\begin{figure}[t]
\vskip 0.2in
	\begin{center}
	\centerline{\includegraphics[width=250pt]{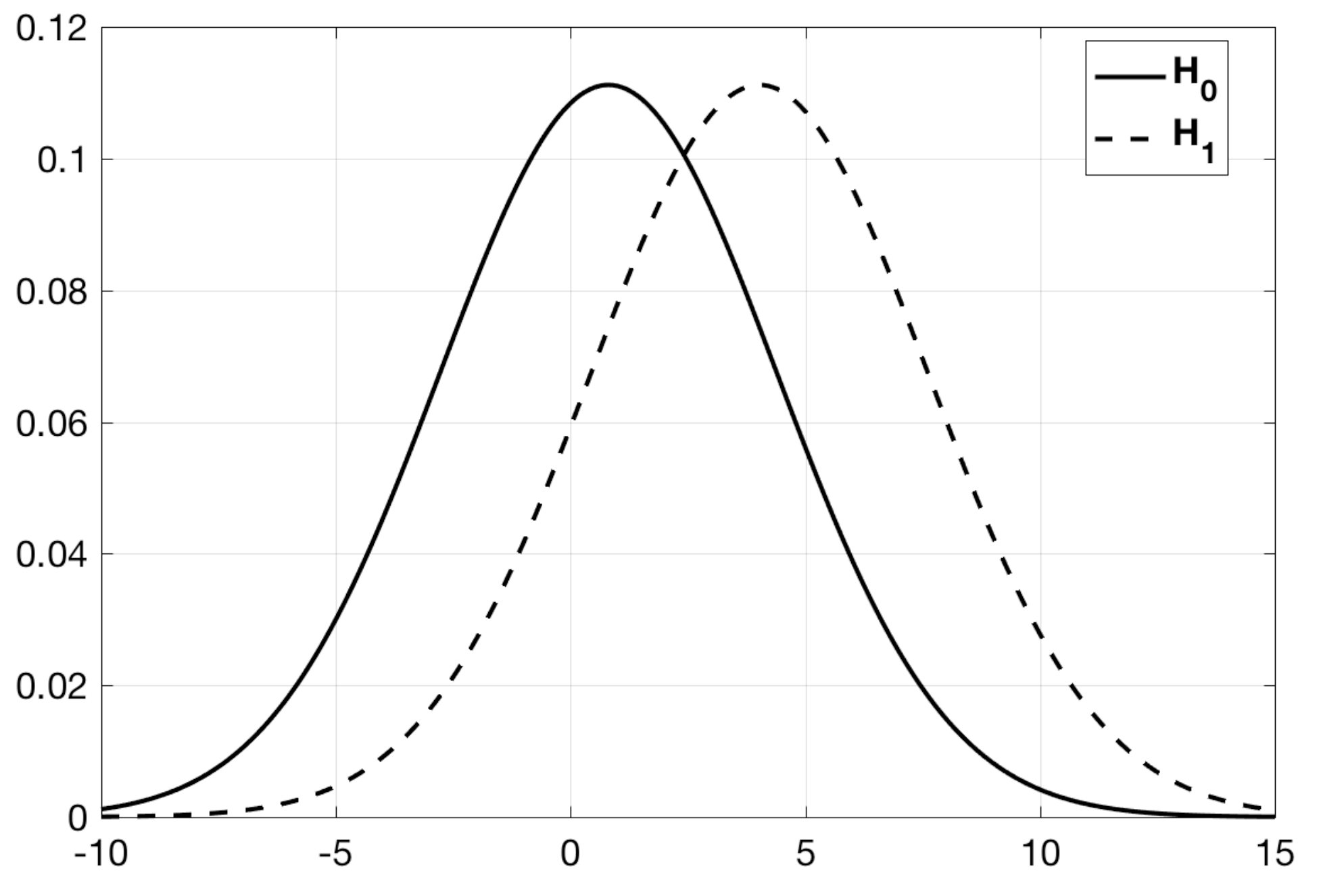}}
	\caption{The limiting density of the test statistic $L_{\SNR}$ in \eqref{eq:test_0} and \eqref{eq:test_1} under $\bsH_0$ (solid) and under $\bsH_1$ (dashed), respectively, with $\SNR=0.8$ for $w_2 = 2$ and $w_4 = 3$ (e.g. GOE noise).}
	\label{fig:normal}
\end{center}
\vskip -0.2in
\end{figure}

Based on Theorem \ref{thm:main}, we propose a hypothesis test described in Algorithm \ref{alg:ht}. In this test, given a data matrix $M$, we compute $L_{\SNR}$ and compare it with the critical value defined as
\beq \begin{split} \label{eq:m_lambda}
	m_{\SNR} := \frac{m_0+m_+}{2} 
	= -\log(1-\SNR) + (w_2-1) \left( \frac{1}{w_4 -1} - \frac{1}{w_2} \right) \SNR 
	+ \left( \frac{w_4}{4} -1 + \frac{1}{2(w_4 -1)} \right) \SNR^2
\end{split} \eeq
to accept or reject the null hypothesis test.

\begin{algorithm}[tb]
	\caption{Hypothesis test}
	\label{alg:ht}

	\KwData{$M_{ij}$, parameters $w_2, w_4$, $\SNR$}
	$L_{\SNR} \gets$ test statistic in \eqref{eq:L_lambda} \;
	$m_{\SNR} \gets$ critical value in \eqref{eq:m_lambda} \;
	\uIf{$L_{\lambda} \leq m_{\SNR}$}{ {\bf Accept} $\bsH_0$ \;}
	\Else{ {\bf Reject} $\bsH_0$ \;}

\end{algorithm}

\begin{thm} \label{thm:test}
The error of the test in algorithm \ref{alg:ht},
\[
	\err(\SNR) = \p( L_{\SNR} > m_{\SNR} | \bsH_0) + \p( L_{\SNR} \leq m_{\SNR} | \bsH_1),
\]
converges to 
\[
	\erfc \left( \frac{1}{4} \sqrt{-\log (1-\SNR) + \left( \frac{2}{w_2} - 1 \right) \SNR + \left( \frac{1}{w_4-1} - \frac{1}{2} \right) \SNR^2} \right)
\]
where $\erfc(\cdot)$ is the complementary error function defined as
\[
	\erfc(z) = \frac{2}{\sqrt{\pi}} \int_z^{\infty} e^{-x^2} \,\dd x \,.
\]
\end{thm}

\begin{proof}
Due to the symmetry, $\p( L_{\SNR} > m_{\SNR} | \bsH_0)$ and $\p( L_{\SNR} \leq m_{\SNR} | \bsH_1)$ converge to a common limit. Since
\[
	\p( L_{\SNR} > m_{\SNR} | \bsH_0) = \p \left( \frac{L_{\SNR} - m_0}{\sqrt{V_0}} > \frac{m_+-m_0}{2\sqrt{V_0}} \bigg| \bsH_0 \right)
\]
and $V_0 = 2 \left( m_+-m_0 \right)$,
we can identify the limit as
\[
	\p \left( Z > \frac{\sqrt{V_0}}{4} \right)
\]
for a standard Gaussian random variable $Z$. Thus, we can conclude that
\beq \label{eq:limiting_error}
	\lim_{N \to \infty} \err(\SNR) = 2\p \left( Z > \frac{\sqrt{V_0}}{4} \right) = \erfc \left( \frac{E_{\SNR}}{4}\right).
\eeq
\end{proof}

\begin{rem}
Even when the exact values of $w_2$ and $w_4$ are not known, we can estimate these parameters from the data matrix by computing $\frac{1}{N} \sum_i M_{ii}^2$ and $\frac{2}{N(N-1)} \sum_{i<j} M_{ij}^2$, respectively. Such estimates are accurate enough for the algorithm as we can easily check from the Chernoff bound.
\end{rem}

In case $w_4=3$ and $w_2=\infty$, we obtain
\beq
	\lim_{N \to \infty} \err(\SNR) = \erfc \left( \frac{1}{4} \sqrt{-\log (1-\SNR) - \SNR} \right),
\eeq
which is equal to the error of the likelihood ratio test, given in Corollary 5 of \cite{AlaouiJordan2018}. Furthermore, in case $w_4=3$ and $w_2 < \infty$, we get
\beq \label{er:optimal_error2}
	\lim_{N \to \infty} \err(\SNR) = \erfc \left( \frac{1}{4} \sqrt{-\log (1-\SNR) - \SNR + \frac{2\SNR}{w_2}} \right),
\eeq
which coincides with the error of the likelihood ratio test, obtained in the remark after Theorem 2 of \cite{AlaouiJordan2018} with $\E_{P_X}[X^3] = 0$. Thus, our test achieves the optimal error when $\lambda$ is below a certain threshold $\lambda_c$ above which reliable detection is possible as shown in \cite{NIPS_Barbier,LelargeMiolane,AlaouiJordan2018}. Since $\lambda_c=1$ in many cases including spherical, Rademacher, and any i.i.d. prior with a sub-Gaussian bound \cite{Perry2018}, and since universal tests such as ours cannot exploit the knowledge on the prior, we have considered a test that works for any $\lambda<1$.

In the extreme cases $w_2=0$ or $w_4=1$, the means and the variance in Theorem \ref{thm:main} are not well-defined. However, it actually means there exists a function $f$ such that the variance $V_M(f)$ vanishes, hence the signal can be detected reliably. See Section \ref{sec:except} for detail.

In applications such as angular synchronization, the signal and the noise are given as complex numbers. It is natural in this case to consider a complex spiked Wigner matrix, which is defined as in Definition \ref{defn:Wigner} with the following modifications:

\begin{itemize}
\item For all $i<j$, the real part and the imaginary part of $H_{ij}$ are independent to each other.
\item For all $i<j$, $\E[H_{ij}^2]=0$, $N\E[|H_{ij}|^2]=1$, $N^{\frac{3}{2}} \E[|H_{ij}|^3]=w_3$, and $N^2 \E[|H_{ij}|^4]=w_4$ for some constants $w_3 \in \R$ and $w_4 > 0$.
\end{itemize}

When the data matrix is a complex spiked Wigner matrix, we construct a test based on the LSS with the function
\beq \begin{split}
	\phi_{\SNR}(x) = \log \left( \frac{1}{1-\sqrt{\SNR} x + \SNR} \right) + \sqrt{\SNR} \left( \frac{1}{w_2} - 1 \right) x + \frac{\SNR}{2} \left( \frac{1}{w_4-1} - 1 \right) x^2,
\end{split} \eeq
and thus the test statistic $L_{\SNR}$ is
\beq \begin{split}
	L_{\SNR} &= \sum_{i=1}^N \phi_{\SNR}(\mu_i) -  N \int_{-2}^2 \frac{\sqrt{4-z^2}}{2\pi} \phi_{\SNR}(z) \, \dd z \\
	&= - \log \det \left( (1+\SNR)I - \sqrt{\SNR} M \right) + \frac{\SNR N}{2} 
	+ \sqrt{\SNR} \left( \frac{1}{w_2} - 1 \right) \Tr M 
	+ \frac{\SNR}{2} \left( \frac{1}{w_4-1} - 1 \right) (\Tr M^2 - N).
\end{split} \eeq
In the test, we compare $L_{\SNR}$ with the critical value $m_{\SNR}$, defined by
\[ \begin{split}
	m_{\SNR} = -\log(1-\SNR) + \frac{w_2-1}{2} \left( \frac{1}{w_4 -1} - \frac{1}{w_2} \right) \SNR 
	+ \left( \frac{w_4-3}{4} + \frac{1}{4(w_4 -1)} \right) \SNR^2.
\end{split} \]
We have the following result for the proposed test with a complex Wigner matrix.

\begin{thm} \label{thm:complex}
Let $M$ be a complex spiked Wigner matrix in Definition \ref{defn:spiked_Wigner} with $\lambda = \SNR \in (0, 1)$. Suppose $w_2>0$ and $w_4>1$. Denote by $\mu_1 \geq \mu_2 \geq \dots \geq \mu_N$ the eigenvalues of $M$. Then, for any spike $\bsx \in \C^N$ with $\| \bsx \|_2 = 1$, $L_{\SNR}$ converges to a Gaussian random variable. The error of the test in algorithm \ref{alg:ht},
\[
	\err(\SNR) = \p( L_{\SNR} > m_{\SNR} | \bsH_0) + \p( L_{\SNR} \leq m_{\SNR} | \bsH_1),
\]
converges to
\[
	\erfc \left( \frac{1}{4} \sqrt{-\log (1-\SNR) + \left( \frac{1}{w_2} - 1 \right) \SNR + \left( \frac{1}{w_4-1} - 1 \right) \frac{\SNR^2}{2}} \right).
\]
\end{thm}

The proof of Theorem \ref{thm:complex} is an almost verbatim copy of the proofs of Theorems \ref{thm:main} and \ref{thm:test} except the change of the mean and the variance of the limiting Gaussian distribution in \ref{thm:main} (See Remark \ref{rem:complex}); we omit the detail.

In case $w_2=1$, $w_4=2$, which corresponds to the Gaussian Unitary Ensemble (GUE), the limiting error is
\[
	\erfc \left( \frac{1}{4} \sqrt{-\log (1-\SNR)} \right),
\]
which is equal to the GOE case where $w_2=2$ and $w_4=3$; see \eqref{er:optimal_error2}.

\section{Adaptive test} \label{sec:adaptive}

In this section, we explain how we can find an adequate candidate of SNR for the test when the presence of the signal is not known, but the prior distribution of SNR under $\bsH_1$ is known. In this case, $\lambda=0$ under $\bsH_0$ but $\lambda$ is drawn from a distribution on $(0, 1)$ under $\bsH_1$. Since the SNR for $\bsH_1$ is not known, we introduce a parameter $t$ that we use as a representative value of SNR under $\bsH_1$ in the test.

We consider the test proposed in Section \ref{sec:main} with SNR $t$, which involves the test statistic $L_{t}$ in \eqref{eq:L_lambda}. Then, for a given $t$, $L_{t}$ converges to a Gaussian distribution with the mean
\beq \begin{split}
	m_{t} (\lambda) &:= -\frac{1}{2} \log(1-t) + \left(\frac{w_2 -1}{w_4-1} -\frac{1}{2} \right)t + \frac{(w_4 -3) t^2}{4} \\
	&\qquad - \log(1-\sqrt{\lambda t}) + \left( \frac{2}{w_2} - 1 \right) \sqrt{\lambda t} + \left( \frac{1}{w_4-1} - \frac{1}{2} \right) \lambda t
\end{split} \eeq
and the variance
\beq
	V_{t} := 2\log \left(\frac{1}{1-t} \right) + \left( \frac{4}{w_2} - 2 \right) t + \left( \frac{2}{w_4-1} - 1 \right) t^2.
\eeq
The test compares $L_t$ with the critical value $m_{t}$ defined in \eqref{eq:m_lambda}.

The probability of Type-I error
\[
	\p( L_t > m_t | \bsH_0) \to \p \left( Z > \frac{m_t(t)-m_t(0)}{2\sqrt{V_t}} \bigg| \bsH_0 \right) = \frac{1}{2} \erfc \left( \frac{m_t(t)-m_t(0)}{2\sqrt{2 V_t}} \right),
\]
where $Z$ is a standard Gaussian random variable. Similarly, the probability of Type-II error
\[
	\p( L_t \leq m_t | \bsH_1) \to \p \left( Z \leq \frac{2m_t(\lambda)-m_t(t)-m_t(0)}{2\sqrt{V_t}} \bigg| \bsH_1 \right) = \frac{1}{2} \erfc \left( \frac{2m_t(\lambda)-m_t(t)-m_t(0)}{2\sqrt{2 V_t}} \right).
\]
The average error then converges to
\beq \label{eq:average_error}
	\frac{1}{2} \erfc \left( \frac{m_t(t)-m_t(0)}{2\sqrt{2 V_t}} \right) + \frac{1}{2} \int_0^1 \erfc \left( \frac{2m_t(\lambda)-m_t(t)-m_t(0)}{2\sqrt{2 V_t}} \right) \dd P_{SNR} (\lambda),
\eeq
where $P_{SNR}$ denotes the prior distribution of SNR $\lambda$ under $\bsH_1$. We may choose $t$ as the minimizer of the right hand side of \eqref{eq:average_error}.

In the simplest case where $\lambda$ is drawn from $\mathrm{unif}(0, 1)$ under $\bsH_1$ and $w_2 = 2$, $w_4 = 3$, the limiting error is
\beq \label{eq:adaptive_error}
	\frac{1}{2} \erfc \left( \frac{-\log (1-t)}{4\sqrt{-\log (1-t)}} \right) + \frac{1}{2} \int_0^1 \erfc \left( \frac{-2 \log (1-\sqrt{\lambda t}) + \log (1-t) }{4\sqrt{-\log (1-t)}} \right) \dd \lambda,
\eeq
which attains its minimum $0.771 \cdots$ when $t=0.671 \cdots$ as shown in Figure \ref{fig:adaptive}. This in particular shows that the test performs better than the random guess whose error is $1$.

\begin{figure}[t]
\vskip 0.2in
	\begin{center}
	\centerline{\includegraphics[width=200pt]{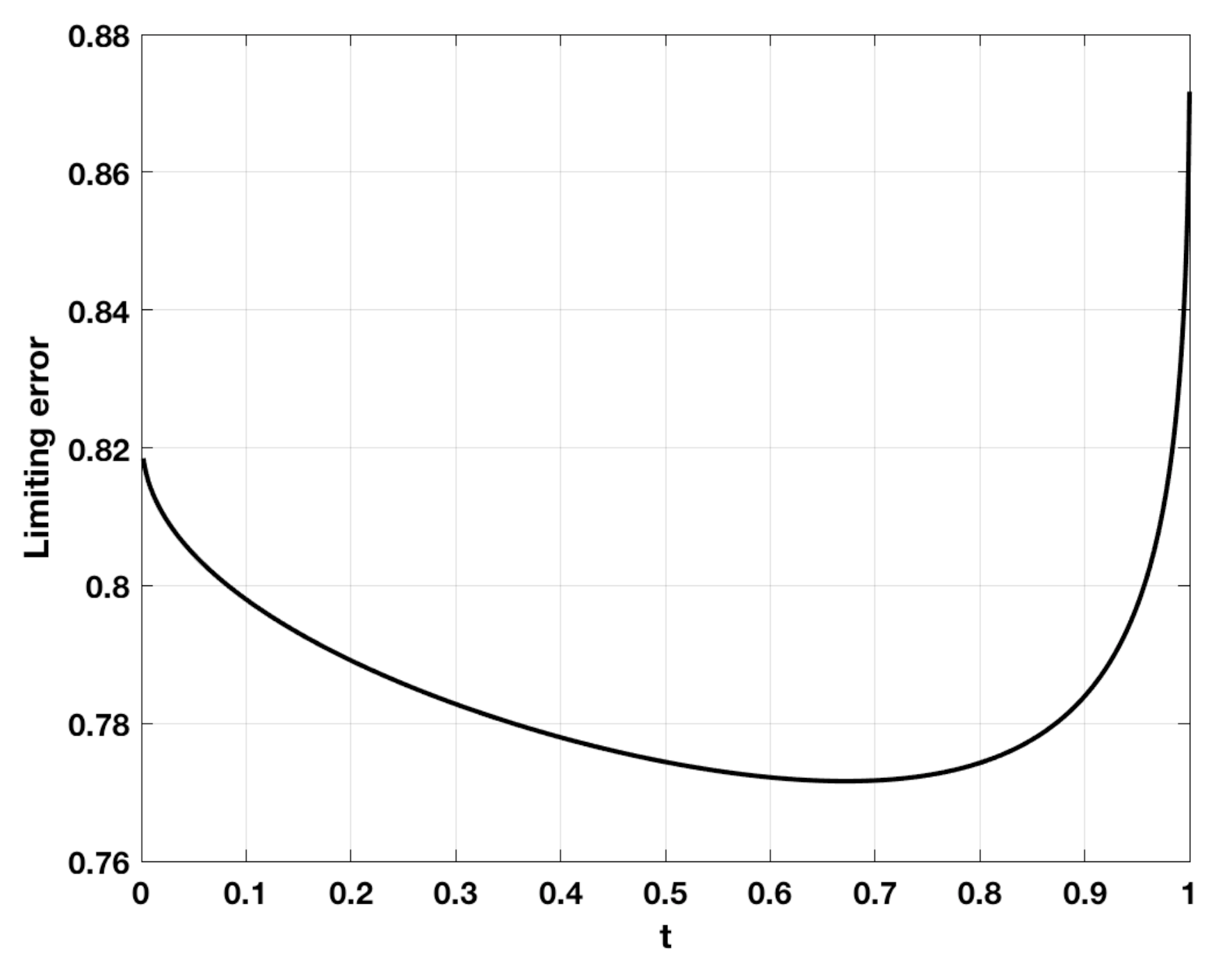}}
	\caption{The limiting error in \eqref{eq:adaptive_error} over the representative value of SNR $t$.}
	\label{fig:adaptive}
\end{center}
\vskip -0.2in
\end{figure}

\section{Test with Entrywise Transformation} \label{sec:trans}

Suppose that each normalized entry $\sqrt{N} H_{ij}$ is drawn from a distribution $\caP$ with a density function $g$. As shown in \cite{Perry2018}, it turns out that the signal can be reliably detected by PCA if $\SNR > 1/\fh$, where $\fh$ is the Fisher information of $\caP$ defined by
\beq
	\fh = \int_{-\infty}^{\infty} \frac{g'(w)^2}{g(w)} \dd w.
\eeq
Since $\fh \geq 1$ with equality if and only if $\caP$ is a standard Gaussian, if $H$ is non-Gaussian, the detection problem becomes easier. 

The main idea of improving the detection threshold for PCA is based on the following entrywise transformation. Set
\[
	h(w) := -\frac{g'(w)}{g(w)}.
\]
Given the data matrix $M$, one can consider a transformed matrix $\tM$ obtained by
\[
	\tM_{ij} = \frac{1}{\sqrt{\fh N}} h(\sqrt{N} M_{ij}) \qquad (i \neq j).
\]
(See \eqref{eq:tM1} for the diagonal entries.) The transformation effectively changes the SNR from $\SNR$ to $\SNR \fh$ for PCA, and thus it is possible to reliably detect the signal if $\SNR \fh > 1$. For more detail, see Section 4 of \cite{Perry2018}.

If $\SNR < 1/\fh$, no tests based on PCA are reliable. Hence, we consider the weak detection of the signal with the entrywise transformation. 
The effective change of the SNR by the entrywise transformation suggests that the result in Theorem \ref{thm:main} will also change correspondingly with the entrywise transformation. For analysis, we will assume the following:

\begin{assump} \label{assump:entry}
For the spike $\bsx$, we assume that $\| \bsx \|_{\infty} \leq N^{-\phi}$ for some $\phi > \frac{3}{8}$. 

For the noise, let $\caP$ and $\caP_d$ be the distributions of the normalized off-diagonal entries $\sqrt{N} H_{ij}$ and the normalized diagonal entries $\sqrt{N} H_{ii}$, respectively. We assume the following:
\begin{enumerate}
\item The density function $g$ of $\caP$ is smooth, positive everywhere, and symmetric (about 0).
\item For any fixed $D$, the $D$-th moment of $\caP$ is finite.
\item The function $h = -g'/g$ and its all derivatives are polynomially bounded in the sense that $|h^{(\ell)}(w)| \leq C_{\ell} |w|^{C_{\ell}}$ for some constant $C_{\ell}$ depending only on $\ell$.
\item The density function $g_d$ of $\caP_d$ satisfies the assumptions 1--3.
\end{enumerate}  
\end{assump}

Note that the signal is not necessarily delocalized, i.e., $\| \bsx \|_{\infty}$ can be significantly larger than $N^{-1/2}$. We remark that the assumption $\| \bsx \|_{\infty} \ll N^{-3/8}$ is a technical constraint and may not be optimal.

Let $h = -g'/g$ and $h_d = -g_d'/g_d$. For a spiked Wigner matrix $M$ in Definition \ref{defn:spiked_Wigner} that satisfies Assumption \ref{assump:entry}, define a matrix $\tM$ by
\beq \label{eq:tM1}
	\tM_{ij} = \frac{1}{\sqrt{\fh N}} h(\sqrt{N} M_{ij}) \quad (i \neq j), \qquad 
	\tM_{ii} = \sqrt{\frac{w_2}{\fh_d N}} h_d \big(\sqrt{\frac{N}{w_2}} M_{ii} \big),
\eeq
where
\[
	\fh = \int_{-\infty}^{\infty} \frac{g'(w)^2}{g(w)} \dd w, \qquad \fh_d = \int_{-\infty}^{\infty} \frac{g_d'(w)^2}{g_d(w)} \dd w.
\]
The transformed matrix $\tM$ is not a spiked Wigner matrix anymore. Nevertheless, as we will prove in Theorem \ref{thm:trans_CLT} in Section \ref{sec:CLT}, the CLT for the LSS of $\tM$ holds with the mean $m_{\tM}(f)$ and the variance $V_{\tM}(f)$. Further, as we will see in Theorem \ref{thm:trans_main}, when compared with Theorem \ref{thm:main}, the parameter $\SNR$ in the mean and the variance of the CLT is replaced by $\SNR \fh$, especially in the logarithmic term. It shows that the entrywise transform effectively increases the SNR from $\SNR$ to $\SNR \fh$.

Denote by $m_{\tM_0}(f)$ the mean $m_{\tM}(f)$ with $\lambda = 0$. Then, as in Section \ref{sec:main}, we need to maximize
\[
	\left| \frac{m_{\tM}(f) - m_{\tM_0}(f)}{\sqrt{V_{\tM}(f)}} \right|.
\]
In Theorem \ref{thm:trans_optimize}, we prove that the maximum is attained if and only if $f(x) = C_1 \wt\phi_{\SNR}(x) + C_2$ for some constants $C_1$ and $C_2$, where
\beq \begin{split} \label{eq:trans_phi}
	\wt\phi_{\SNR}(x) := \log \left( \frac{1}{1-\sqrt{\SNR\fh} x + \SNR\fh} \right) 
	+ \sqrt{\SNR} \left( \frac{2\sqrt{\fh_d}}{w_2} - \sqrt{\fh} \right) x 
	+ \SNR \left( \frac{\gh}{\wt{w_4}-1} - \frac{\fh}{2} \right) x^2
\end{split} \eeq
with
\[
	\gh = \frac{1}{2\fh} \int_{-\infty}^{\infty} \frac{g'(w)^2 g''(w)}{g(w)^2} \dd w,
	\quad \wt{w_4} = \frac{1}{(\fh)^2} \int_{-\infty}^{\infty} \frac{(g'(w))^4}{(g(w))^3} \dd w.
\]
Thus, denoting by $\wt\mu_1 \geq \wt\mu_2 \geq \dots \geq \wt\mu_N$ the eigenvalues of $\tM$, we define the test statistic $\wt L_{\SNR}$ by
\beq \begin{split} \label{eq:wt L_lambda}
	\wt L_{\SNR} &:= \sum_{i=1}^N \wt\phi_{\SNR}(\wt\mu_i) -  N \int_{-2}^2 \frac{\sqrt{4-z^2}}{2\pi} \phi_{\SNR}(z) \, \dd z \\
	&= - \log \det \left( (1+\SNR\fh)I - \sqrt{\SNR\fh} \tM \right) + \frac{\SNR\fh}{2} N \\
	&\qquad \qquad + \sqrt{\SNR} \left( \frac{2\sqrt{\fh_d}}{w_2} - \sqrt{\fh} \right) \Tr \tM 
	+ \SNR \left( \frac{\gh}{\wt{w_4}-1} - \frac{\fh}{2} \right) (\Tr \tM^2 - N).
\end{split} \eeq

The CLT for $\wt L_{\SNR}$ holds as follows:

\begin{thm} \label{thm:trans_main}
Let $M$ be a spiked Wigner matrix in Definition \ref{defn:spiked_Wigner} that satisfy Assumption \ref{assump:entry}. 
Suppose that $\SNR < \frac{1}{\fh}$. 
Then,
\[
	\wt L_{\SNR} \Rightarrow \caN(\wt m_0, \wt V_0) \qquad \text{ if } \lambda = 0
\]
and
\[
	\wt L_{\SNR} \Rightarrow \caN(\wt m_+, \wt V_0) \qquad \text{ if } \lambda = \SNR > 0.
\]
The mean and the variance of the limiting Gaussian distribution are given by
\[ \begin{split}
	\wt m_0 = - \frac{1}{2} \log(1-\SNR\fh) 
	+ \left( \frac{(w_2 -1)\gh}{\wt w_4 -1} - \frac{\fh}{2} \right) \SNR + \frac{\wt w_4 -3}{4} (\SNR\fh)^2,
\end{split} \]
\[ \begin{split}
	\wt m_+ = \wt m_0 - \log(1-\SNR\fh) 
	+ \left( \frac{2\fh_d}{w_2} - \fh \right) \SNR + \left( \frac{(\gh)^2}{\wt w_4-1} - \frac{(\fh)^2}{2}\right) \SNR^2,
\end{split} \]
and
\[ \begin{split}
	\wt V_0 = -2 \log(1-\SNR\fh) + \left( \frac{4\fh_d}{w_2} - 2\fh \right) \SNR 
	+ \left( \frac{2(\gh)^2}{\wt w_4-1} - (\fh)^2 \right) \SNR^2 \,.
\end{split} \]
\end{thm}

Theorem \ref{thm:trans_main} is a direct consequence of a general CLT in Theorem \ref{thm:trans_CLT} in Section \ref{sec:CLT}.

With the entrywise transformation, we modify the hypothesis test as in Algorithm \ref{alg:htet}, where we compute $\wt L_{\SNR}$ and compare it with 
\beq \begin{split} \label{eq:wt m_lambda}
	\wt m_{\SNR} &:= \frac{\wt m_0 + \wt m_+}{2} \\
	&= -\log(1-\SNR\fh) 
	+ \left( \frac{\fh_d}{w_2} - \fh + \frac{(w_2 -1) \gh}{\wt{w_4}-1} \right) \SNR 
	+ \left( \frac{\wt{w_4}}{4} -1 \right) (\SNR \fh)^2 + \frac{(\SNR\gh)^2}{2(\wt{w_4}-1)}.
\end{split} \eeq

\begin{thm} \label{thm:trans_test}
The error of the test in Algorithm \ref{alg:htet},
\[
	\err(\SNR) = \p( \wt L_{\SNR} > \wt m_{\SNR} | \bsH_0) + \p( \wt L_{\SNR} \leq \wt m_{\SNR} | \bsH_1),
\]
converges to $\erfc ( \wt E_{\SNR}/4 )$, where
\[ \begin{split}
	\wt E_{\SNR}^2 = \log \left( \frac{1}{1-\SNR\fh} \right) + \left( \frac{2\fh_d}{w_2} - \fh \right) \SNR 
	+ \left( \frac{(\gh)^2}{\wt{w_4}-1} - \frac{(\fh)^2}{2} \right) \SNR^2.
\end{split} \]
\end{thm}

The proof closely follows the proof of Theorem \ref{thm:test}, and we omit the detail.

\begin{algorithm}[tb]
	\caption{Hypothesis test}
	\label{alg:htet}

	\KwData{$M_{ij}$, parameters $w_2, w_4$, $\SNR$, densities $g, g_d$}
	$\tM \gets$ transformed matrix in Equations \eqref{eq:tM1} \; 
	$\wt L_{\SNR} \gets$ test statistic in \eqref{eq:wt L_lambda} \;
	$\wt m_{\SNR} \gets$ critical value in \eqref{eq:wt m_lambda} \;
	\uIf{$\wt L_{\SNR} \leq \wt m_{\SNR}$}{ {\bf Accept} $\bsH_0$ \;}
	\Else{ {\bf Reject} $\bsH_0$ \;}

\end{algorithm}

\section{Examples and Simulations} \label{sec:simul}

We conduct some simulations to numerically check the accuracy of the proposed tests in Section \ref{sec:main} and Section \ref{sec:trans} under various settings.

\subsection{Gaussian noise} \label{subsec:GOE}

We first consider the case where the noise matrix $H$ is a GOE matrix and the signal $\bsx = (x_1, x_2, \dots, x_N)$ where $\sqrt{N} x_i$'s are i.i.d. Rademacher random variable. Let the data matrix $M = \sqrt{\SNR} \bsx \bsx^T + H$. The parameters are $w_2 = 2$ and $w_4 = 3$. 

In the numerical simulation done in Matlab, we generated 10,000 independent samples of the $256 \times 256$ data matrix $M$ under $\bsH_0$ (without signal) and $\bsH_1$ (with signal), respectively, varying SNR $\SNR$ from $0$ to $0.7$. To apply Algorithm \ref{alg:ht} proposed in Section \ref{sec:main}, we computed
\beq \begin{split}
	L_{\SNR} = -\log \det \big( (1+\SNR)I - \sqrt{\SNR} M \big) + \frac{\SNR N}{2},
\end{split} \eeq
and accepted $\bsH_0$ if $L_{\SNR} \leq -\log (1-\SNR)$ and rejected $\bsH_0$ otherwise. The limiting error of the test is
\beq \label{eq:limit_error_1a}
	\erfc \left( \frac{1}{4} \sqrt{-\log (1-\SNR)} \right).
\eeq

\begin{figure}[t]
\vskip 0.2in
	\begin{center}
    \centering
    \subfloat{{\includegraphics[width=0.4\linewidth]{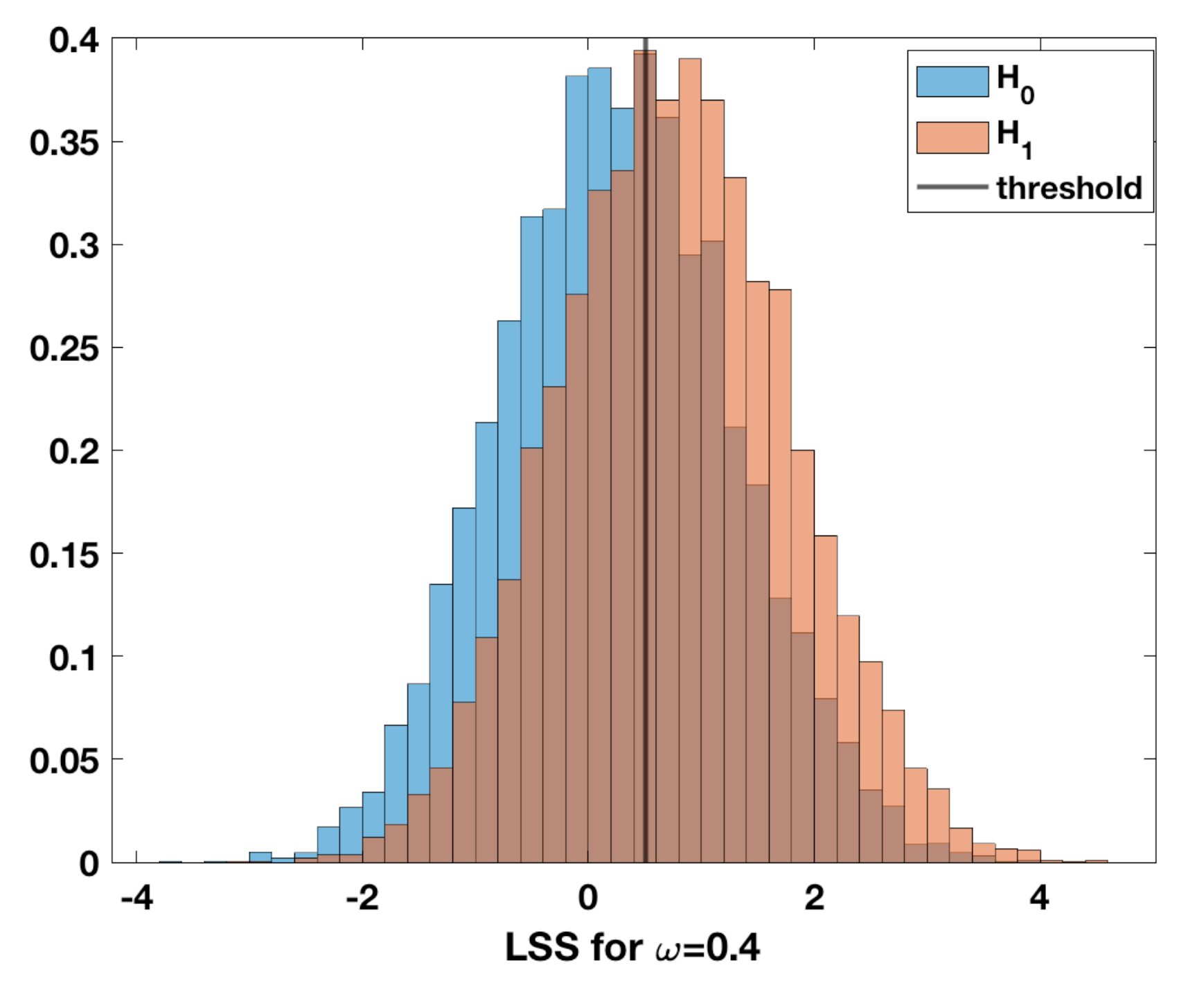}}}
    \qquad
    \subfloat{{\includegraphics[width=0.4\linewidth]{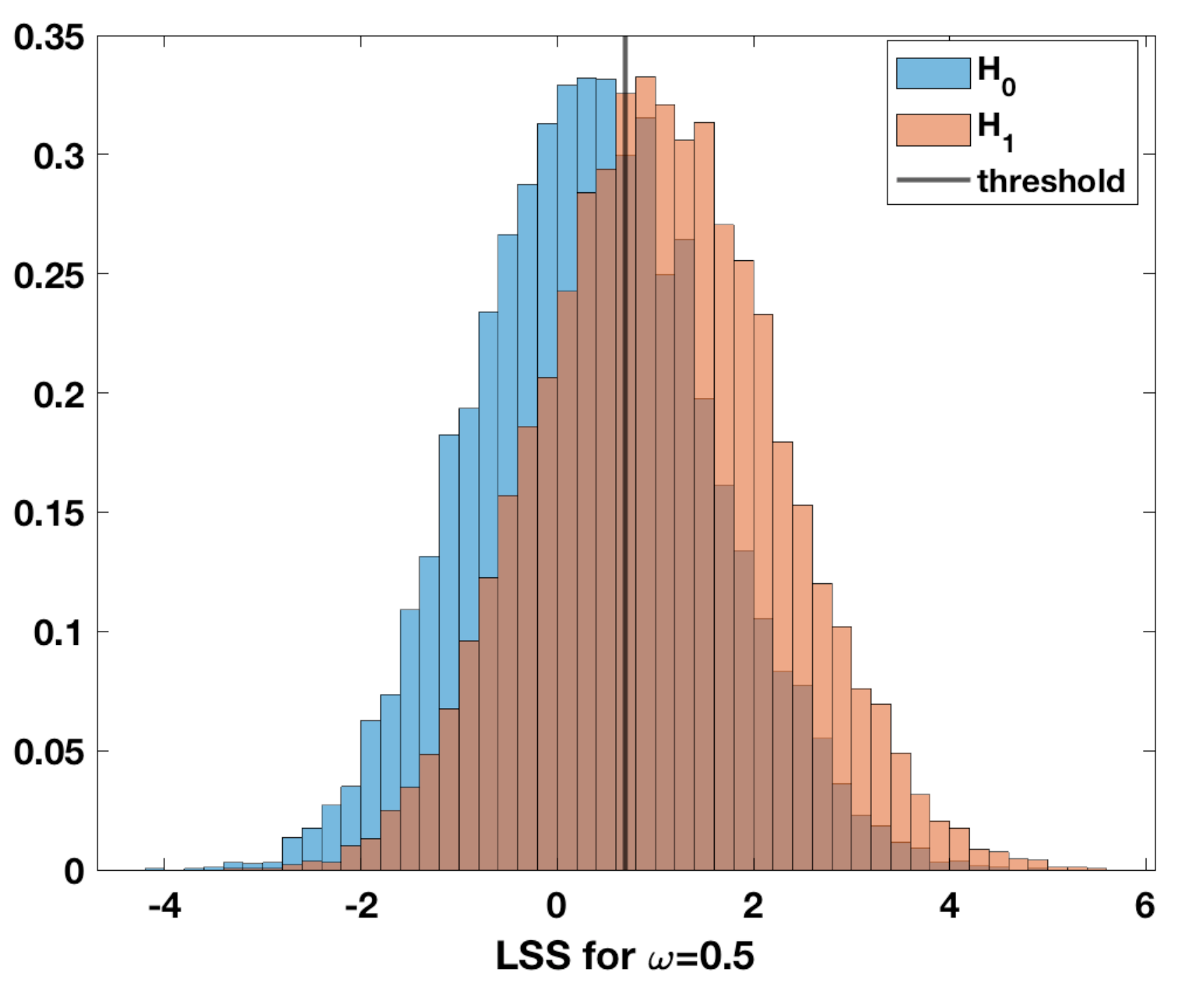}}}
	\caption{The histograms of the test statistic $L_{\SNR}$ under $\bsH_0$ and under $\bsH_1$, respectively, for the setting in Section \eqref{subsec:GOE} with $\SNR = 0.4$ and $\SNR = 0.5$.}
	\label{fig:LSS_GOE}
\end{center}
\vskip -0.2in
\end{figure}

In Figure \ref{fig:LSS_GOE}, we plot the histograms of the test statistic $L_{\SNR}$ under $\bsH_0$ and under $\bsH_1$, respectively, with the test threshold $-\log (1-\SNR)$ for $\SNR = 0.4$ and $\SNR = 0.5$. It can be shown that the difference of the means of $L_{\SNR}$ under $\bsH_0$ and under $\bsH_1$ is larger for $\SNR = 0.5$.
In Figure \ref{fig:GOE}, we plot empirical average (after 10,000 Monte Carlo simulations) of the error of test by Algorithm \ref{alg:ht} and the theoretical error in \eqref{eq:limit_error_1a}. It can be shown that the error of the test closely matches the theoretical error.

\begin{figure}[t]
\vskip 0.2in
	\begin{center}
	\centerline{\includegraphics[width=250pt]{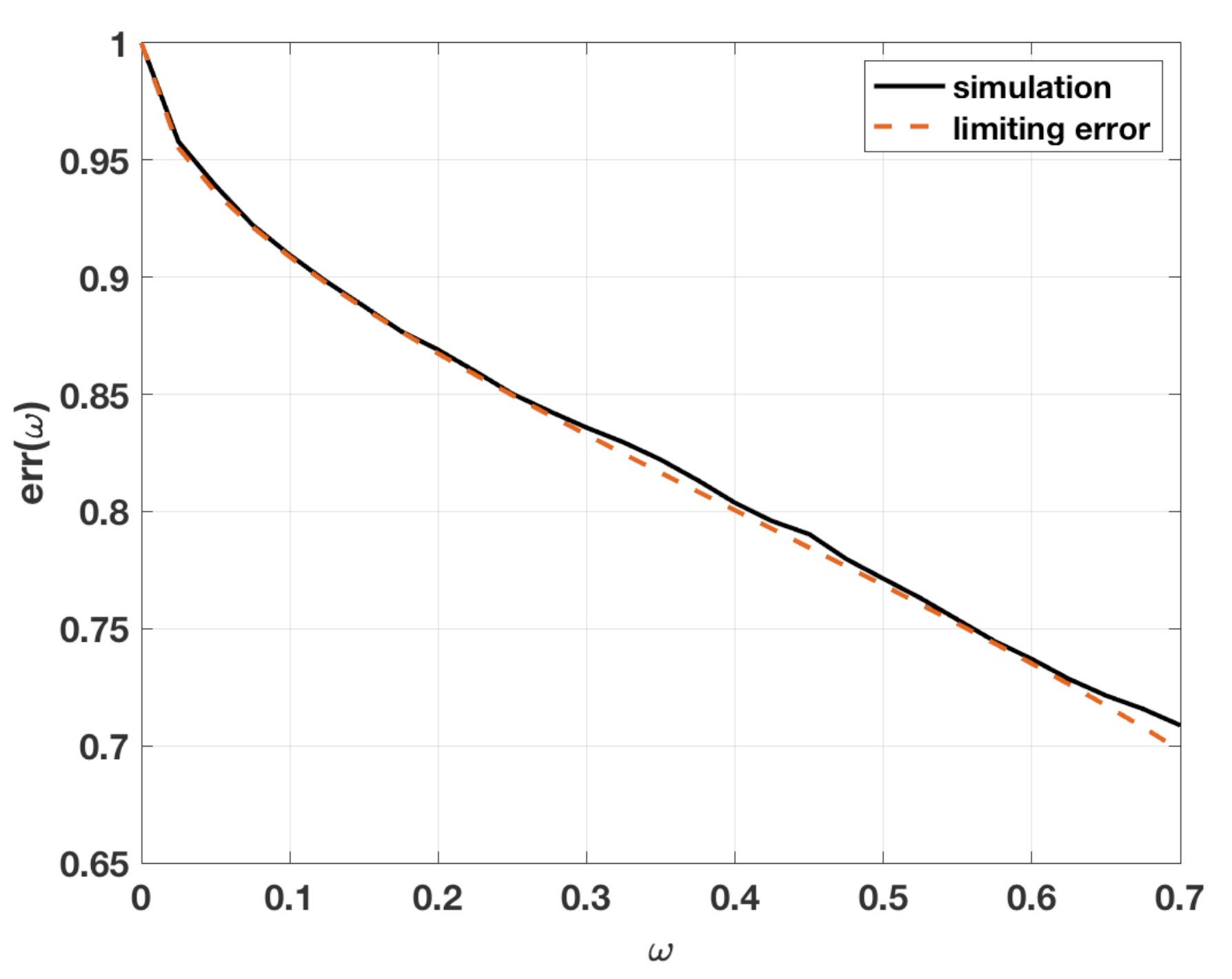}}
	\caption{The error from the simulation (solid) and the theoretical limiting error in \eqref{eq:limit_error_1a} (dashed), respectively, for the setting in Section \ref{subsec:GOE}.}
	\label{fig:GOE}
\end{center}
\vskip -0.2in
\end{figure}

\subsection{Non-Gaussian noise} \label{subsec:sech}

We next consider the case where the density function of the noise matrix is given by
\[
	g(x) = g_d(x) = \frac{1}{2 \cosh (\pi x/2)} = \frac{1}{e^{\pi x/2} + e^{-\pi x/2}}.
\]
Sample $W_{ij} = W_{ji}$ from the density $g$ 
and let $H_{ij} = W_{ij}/\sqrt{N}$. Let $\bsx = (x_1, x_2, \dots, x_N)$ where $\sqrt{N} x_i$'s are i.i.d. Rademacher random variable. Let the data matrix 
$M = \sqrt{\lambda} \bsx \bsx^T + H$. The parameters are $w_2 = 1$ and $w_4 = 5$. We perform the numerical simulation 10,000 samples of the $256 \times 256$ data matrix $M$ with and without the signal, respectively, varying SNR $\SNR$ from $0$ to $0.7$.

In Algorithm \ref{alg:ht} proposed in Section \ref{sec:main}, we compute 
\beq \begin{split}
	L_{\SNR} = -\log \det \big( (1+\SNR)I - \sqrt{\SNR} M \big) + \frac{\SNR N}{2} 
	+ \sqrt{\SNR} \Tr M - \frac{\SNR}{4} (\Tr M^2 - N),
\end{split} \eeq
and accept $\bsH_0$ if $L_{\SNR} \leq -\log (1-\SNR) + \frac{3\SNR^2}{8}$ and reject $\bsH_0$ otherwise. The limiting error of the test is
\beq \label{eq:limit_error_1}
	\erfc \left( \frac{1}{4} \sqrt{-\log (1-\SNR) + \SNR - \frac{\SNR^2}{4}} \right).
\eeq

We can further improve the test by introducing the entrywise transformation given by
\[
	h(x) = -\frac{g'(x)}{g(x)} = \frac{\pi}{2} \tanh \frac{\pi x}{2}.
\]
The Fisher information $\fh$ is $\frac{\pi^2}{8}$, which is strictly larger than $1$. We first construct a pre-transformed matrix $\wt M$ by
\[
	\wt M_{ij} = \frac{2\sqrt{2}}{\pi \sqrt{N}} h(\sqrt{N} M_{ij}) = \sqrt{\frac{2}{N}} \tanh \left( \frac{\pi \sqrt{N}}{2} M_{ij} \right).
\]
If $\SNR > \frac{1}{\fh} = \frac{8}{\pi^2}$, we can use PCA to reliably detect the signal. If $\SNR < \frac{8}{\pi^2}$, we compute the test statistic
\[ \begin{split}
	\wt L_{\SNR} = -\log \det \left( (1+\frac{\pi^2 \SNR}{8})I - \sqrt{\frac{\pi^2 \SNR}{8}} \tM \right) + \frac{\pi^2 \SNR N}{16} 
	+ \frac{\pi \sqrt{\SNR}}{2 \sqrt{2}} \Tr \tM + \frac{\pi^2 \SNR}{16} (\Tr \tM^2 - N).
\end{split} \]
(Here, $\fh = \fh_d = \frac{\pi^2}{8}$, $\gh = \frac{\pi^2}{16}$ and $\wt w_4 = \frac{3}{2}$.) We accept $\bsH_0$ if 
\[
	\wt L_{\SNR} \leq -\log \left(1-\frac{\pi^2 \SNR}{8} \right) - \frac{3 \pi^4 \SNR^2}{512}
\]
and reject $\bsH_0$ otherwise. The limiting error with entrywise transformation is
\beq \label{eq:limit_error_2}
	\erfc \left( \frac{1}{4} \sqrt{ -\log (1-\frac{\pi^2 \SNR}{8}) + \frac{\pi^2 \SNR}{8}} \right).
\eeq
Since $\erfc(z)$ is a decreasing function of $z$ and $\frac{\pi^2}{8} > 1$, it is direct to see that the limiting error in \eqref{eq:limit_error_2} is strictly less than the limiting error in \eqref{eq:limit_error_1}.

The result of the simulation can be seen from Figure \ref{fig:alg12}, which shows that the error from Algorithm \ref{alg:htet} is smaller than that of Algorithm \ref{alg:ht}, and both errors matches theoretical errors in \eqref{eq:limit_error_2} and \eqref{eq:limit_error_1}.

\begin{figure}[t]
\vskip 0.2in
	\begin{center}
	\centerline{\includegraphics[width=250pt]{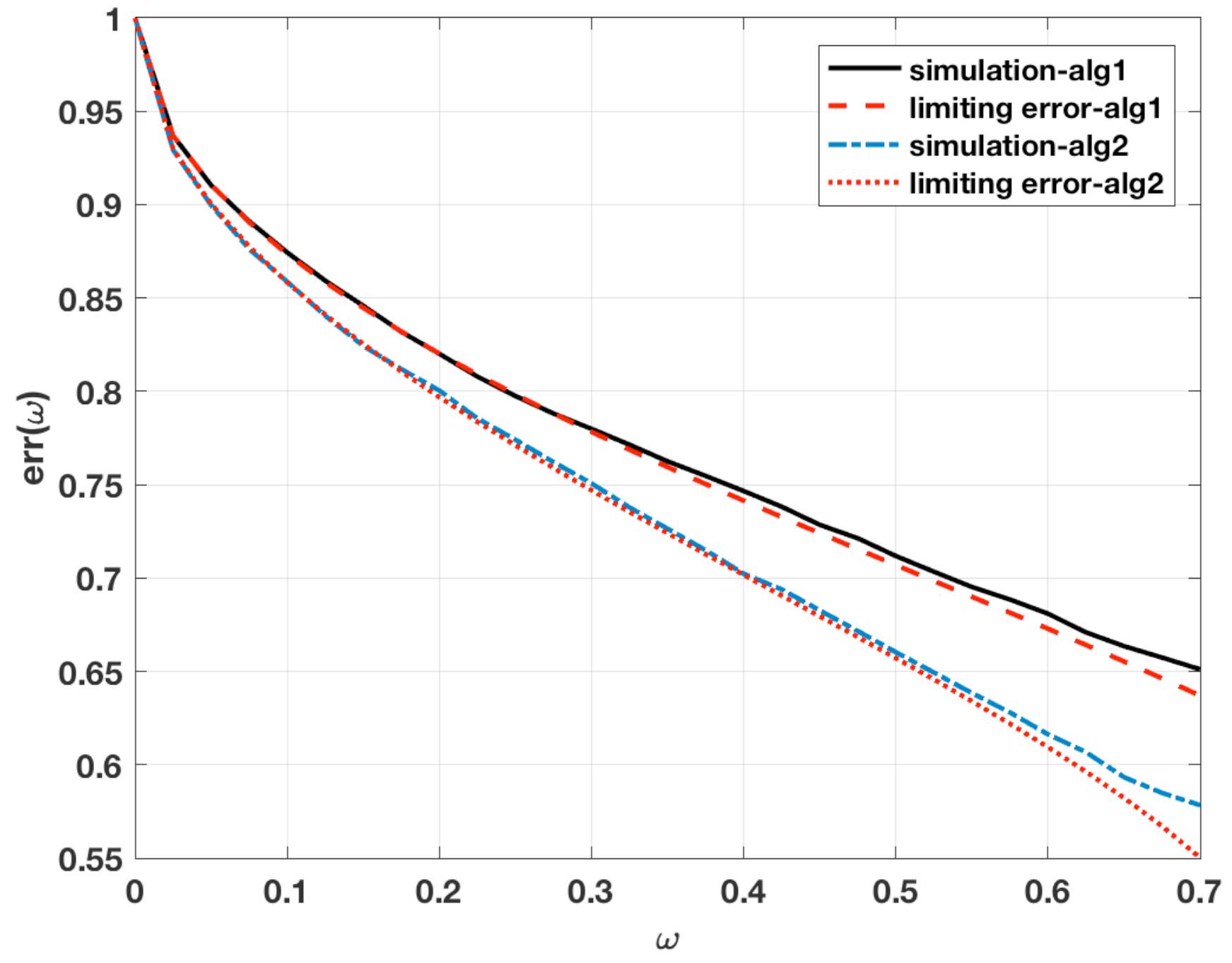}}
	\caption{The errors from the simulation with Algorithm \ref{alg:ht} (black, solid) and with Algorithm \ref{alg:htet} (blue, dash-dot), respectively, versus the limiting errors \eqref{eq:limit_error_1} of Algorithm \ref{alg:ht} (red, dashed) and \eqref{eq:limit_error_2} of Algorithm \ref{alg:htet} (red, dotted), respectively, for the setting in Section \ref{subsec:sech}.}
	\label{fig:alg12}
\end{center}
\vskip -0.2in
\end{figure}

\section{Central Limit Theorems} \label{sec:CLT}

In this section, we present our results on general CLTs for the LSS. The mean and the variance will be written in terms of Chebyshev polynomials (of the first kind) for which we use the following definition.

\begin{defn}[Chebyshev polynomial]
The $n$-th Chebyshev polynomial $T_n$ is a degree $n$ polynomial defined inductively by $T_0(x) = 1$, $T_1(x) = x$, and
\[
	T_{n+1}(x) = 2x T_n(x) - T_{n-1}(x).
\]
It can also be defined by the orthogonality condition
\[
	\int_{-1}^1 T_m(x) T_n(x) \frac{\dd x}{\sqrt{1-x^2}} = \begin{cases}
	0 & \text{ if } m \neq n, \\
	\pi & \text{ if } m=n=0, \\
	\frac{\pi}{2} & \text{ if } m=n\neq 0.
	\end{cases}
\]
\end{defn}

Our first result in this section is the CLT for the LSS with a general function $f$.

\begin{thm} \label{thm:CLT}
Assume the conditions in Theorem \ref{thm:main}. Denote by $\mu_1 \geq \mu_2 \geq \dots \geq \mu_N$ the eigenvalues of $M$. For any function $f$ analytic on an open interval containing $[-2, 2]$,
\[ \begin{split}
	&\left( \sum_{i=1}^N f(\mu_i) - N \int_{-2}^2 \frac{\sqrt{4-z^2}}{2\pi} f(z) \, \dd z \right) \\
	&\qquad \Rightarrow \caN(m_M(f), V_M(f))\,.
\end{split} \]
The mean and the variance of the limiting Gaussian distribution are given by
\[ \begin{split}
	m_M(f) = \frac{1}{4} \left( f(2) + f(-2) \right) -\frac{1}{2} \tau_0(f) + (w_2 -2) \tau_2(f) 
	+ (w_4-3) \tau_4(f) + \sum_{\ell=1}^{\infty} \sqrt{\lambda^{\ell}} \tau_{\ell}(f)
\end{split} \]
and
\[ \begin{split}
	V_M(f) = (w_2-2) \tau_1(f)^2 + 2(w_4-3) \tau_2(f)^2 
	+ 2\sum_{\ell=1}^{\infty} \ell \tau_{\ell}(f)^2,
\end{split} \]
where
\[
	\tau_{\ell}(f) = \frac{1}{\pi} \int_{-2}^2 T_{\ell} \left( \frac{x}{2} \right) \frac{f(x)}{\sqrt{4-x^2}} \dd x
\]
with the $\ell$-th Chebyshev polynomial $T_{\ell}$.
\end{thm}

We prove Theorem \ref{thm:CLT} in Section \ref{sec:CLT_proof}.

\begin{rem} \label{rem:complex}
If $M$ is a complex complex spiked Wigner matrix, Theorem \ref{thm:CLT} holds with the following change:
\[
	m_M(f) = (w_2 -1) \tau_2(f) + (w_4-2) \tau_4(f) + \sum_{\ell=1}^{\infty} \sqrt{\lambda^{\ell}} \tau_{\ell}(f)
\]
and
\[ \begin{split}
	V_M(f) = (w_2-1) \tau_1(f)^2 + 2(w_4-2) \tau_2(f)^2 
	+ \sum_{\ell=1}^{\infty} \ell \tau_{\ell}(f)^2.
\end{split} \]
\end{rem}

Recall that $m_H(f) = m_M(f)$ if $\lambda = 0$. Our second result classifies all functions that are optimal for the hypothesis test.

\begin{thm} \label{thm:optimize}
Assume the conditions in Theorem \ref{thm:CLT}. Let $m_0$ and $m_+$ be as in \eqref{eq:m_0} and \eqref{eq:m_+}, respectively. If $w_2 > 0$ and $w_4 > 1$, then
\beq \label{eq:upper_bound}
	\left| \frac{m_M(f) - m_H(f)}{\sqrt{V_M(f)}} \right| \leq \left| \frac{m_+ - m_0}{\sqrt{V_0}} \right|.
\eeq
The equality holds if and only if $f = C_1 \phi_{\SNR} + C_2$ for some constants $C_1$ and $C_2$ with the function $\phi_{\SNR}$ defined in \eqref{eq:phi}.
\end{thm}

We prove Theorem \ref{thm:optimize} in Section \ref{sec:optimal}.

The function of the form $\phi_{\SNR}$ in \eqref{eq:phi} was considered by Banerjee and Ma for hypothesis testing in stochastic block models; see Remark 3.3 in \cite{BanerjeeMa2017}. Instead of using polynomial approximation of $\phi_{\SNR}$ as in \cite{BanerjeeMa2017}, we use $\phi_{\SNR}$ itself since it is analytic for any $x$ in an open interval $(-\sqrt{\SNR} - \frac{1}{\sqrt{\SNR}}, \sqrt{\SNR} + \frac{1}{\sqrt{\SNR}})$, which contains $[-2, 2]$. In the signal detection test we consider, if there is an eigenvalue outside the interval $(-\sqrt{\SNR} - \frac{1}{\sqrt{\SNR}}, \sqrt{\SNR} + \frac{1}{\sqrt{\SNR}})$, it implies that the signal is present with high probability.

Our result for the pre-transformed CLT is the following theorem:
\begin{thm} \label{thm:trans_CLT}
Assume the conditions in Theorem \ref{thm:trans_main}. For any function $f$ analytic on an open interval containing $[-2, 2]$,
\[ \begin{split}
	\left( \sum_{i=1}^N f(\wt\mu_i) - N \int_{-2}^2 \frac{\sqrt{4-z^2}}{2\pi} f(z) \, \dd z \right)
	\Rightarrow \caN(m_{\tM}(f), V_{\tM}(f))\,.
\end{split} \]
The mean and the variance of the limiting Gaussian distribution are given by
\beq \begin{split} \label{eq:mean_tM}
	m_{\tM}(f) &= \frac{1}{4} \left( f(2) + f(-2) \right) -\frac{1}{2} \tau_0(f) + \sqrt{\SNR\fh_d} \tau_1 (f) 
	+ (w_2 -2 +\SNR \gh) \tau_2(f) + (\wt{w_4}-3) \tau_4(f) \\
	&\quad + \sum_{\ell=3}^{\infty} \sqrt{(\SNR\fh)^{\ell}} \tau_{\ell}(f)
\end{split} \eeq
and
\[ \begin{split}
	V_{\tM}(f) = V_M(f) = (w_2-2) \tau_1(f)^2 + 2(\wt{w_4}-3) \tau_2(f)^2 
	+ 2\sum_{\ell=1}^{\infty} \ell \tau_{\ell}(f)^2.
\end{split} \]
\end{thm}

The proof of Theorem \ref{thm:trans_CLT} is contained in Appendix \ref{sec:entry}.

Let $m_{\tM_0}(f)$ be $m_{\tM}(f)$ in \eqref{eq:mean_tM} with $\lambda = 0$. For the transformed matrix $\tM$, we have the following result that corresponds to Theorem \ref{thm:optimize}.

\begin{thm} \label{thm:trans_optimize}
Assume the conditions in Theorem \ref{thm:trans_CLT}. Then
\beq \label{eq:trans_upper_bound}
	\left| \frac{m_{\tM}(f) - m_{\tM_0}(f)}{\sqrt{V_{\tM}(f)}} \right| \leq \left| \frac{\wt m_+ - \wt m_0}{\sqrt{\wt V_0}} \right|.
\eeq
Here, the equality holds if and only if $f(x) = C_1 \wt\phi_{\SNR}(x) + C_2$ for some constants $C_1$ and $C_2$ with the function $\wt\phi_{\SNR}$ defined in \eqref{eq:trans_phi}.
\end{thm}

See Section \ref{sec:optimal} for the proof of Theorem \ref{thm:trans_optimize}.

\subsection{Exceptional cases} \label{sec:except}

In this subsection, we examine exceptional cases and introduce a feasible test statistic for each case.

\subsubsection*{Exceptional case 1: $w_2=0$} \label{subsubsec:ex1}

In this case, if $\tau_1(f) > 0$ and $\tau_{\ell}(f) = 0$ for all $\ell \geq 2$, then $V_M(f)=0$. It corresponds to choosing $f(x) = x$, and the test statistic is $\Tr M$. Since $w_2=0$, the diagonal entries $H_{ii}$ vanish, hence
\beq
	\Tr M = \sum_{i=1}^N \sqrt{\lambda} x_i^2 = \sqrt{\lambda}
\eeq
from which we can recover $\lambda$.

\subsubsection*{Exceptional case 2: $w_4=1$} \label{subsubsec:ex2}

In this case, if $\tau_1(f) = 0$, $\tau_2(f) > 0$, and $\tau_{\ell} = 0$ for all $\ell \geq 3$, then $V_M(f)=0$. It corresponds to choosing $f(x) = x^2$, and the test statistic is $\Tr M^2 = \sum_{i, j} (M_{ij})^2$. Since $w_4=1$, the off-diagonal entries $H_{ij}$ are Bernoulli random variables, hence
\beq \begin{split}
	\Tr M^2 &= \sum_{i, j} (M_{ij})^2 = \sum_i (H_{ii} + \sqrt{\lambda} x_i^2 )^2 + \sum_{i \neq j} (H_{ij} + \sqrt{\lambda} x_i x_j)^2 \\
	&= w_2 + \sum_i \left( (H_{ii})^2 - \frac{w_2}{N} \right) + (N-1) + \sum_{i \neq j} \left( (H_{ij})^2 - \frac{1}{N} \right) \\
	&\qquad + 2\sqrt{\lambda} \sum_{i, j} H_{ij} x_i x_j + \lambda \sum_{i, j} x_i^2 x_j^2 \\
	&= N-1 +w_2 + \lambda + \caO(N^{-\frac{1}{2}}).
\end{split} \eeq
Thus, we can recover $\lambda$ by computing $\Tr M^2 - (N-1+w_2)$.

\subsubsection*{Exceptional case 3: Biased spike} \label{subsubsec:ex3}

In this last exceptional case, we briefly consider a case that the signal can be reliably detected under a priori information on it. If the signal has a bias, i.e.,
\beq \label{eq:bias}
	\left| \sum_i x_i \right| = c \sqrt{N} (1+o(1))
\eeq
for some $N$-independent constant $c >0$, then we can consider the test statistic
\beq
	\sum_{i, j} M_{ij} = \sum_{i, j} H_{ij} + \sqrt{\lambda} \left( \sum_i x_i \right)^2 = c^2 \sqrt{\lambda} N + o(N),
\eeq
which can be easily checked by applying Chernoff's bound. Note that the condition in \eqref{eq:bias} is satisfied if $\sqrt{N} x_i$'s are independent random variables with mean $c$. We also remark that the test is not based on the spectrum of $M$.

\subsection{Proof of Theorem \ref{thm:optimize} and Theorem \ref{thm:trans_optimize}} \label{sec:optimal}

In this subsection, we prove Theorem \ref{thm:optimize} by applying Theorem \ref{thm:CLT}. The proof of Theorem \ref{thm:trans_optimize} is omitted since it is exactly same as the proof of Theorem \ref{thm:optimize} except that we use Theorem \ref{thm:trans_CLT} instead of Theorem \ref{thm:CLT}.

First, we notice that
\beq
	m_M(f) - m_H(f) = \sum_{\ell=1}^{\infty} \sqrt{\SNR^{\ell}} \tau_{\ell}(f).
\eeq
Recall that
\beq \begin{split}
	V_M(f) &= (w_2-2) \tau_1(f)^2 + 2(w_4-3) \tau_2(f)^2 + 2\sum_{\ell=1}^{\infty} \ell \tau_{\ell}(f)^2 \\
	&= w_2 \tau_1(f)^2 + 2(w_4-1) \tau_2(f)^2 + 2\sum_{\ell=3}^{\infty} \ell \tau_{\ell}(f)^2.
\end{split} \eeq
Assuming $w_2 > 0$ and $w_4 > 1$, by Cauchy's inequality, we obtain that
\beq \label{eq:Cauchy_ineq}
	|m_M(f) - m_H(f)|^2 \leq \left( \frac{\SNR}{w_2} + \frac{\SNR^2}{2(w_4-1)} + \sum_{\ell=3}^{\infty} \frac{\SNR^\ell}{2\ell} \right) V_M(f).
\eeq
From the identity $\log(1-\SNR) = -\sum_{\ell=1}^{\infty} \SNR^{\ell}/\ell$, we get
\beq
	\frac{|m_M(f) - m_H(f)|^2}{V_M(f)} \leq  \frac{\SNR}{w_2} + \frac{\SNR^2}{2(w_4-1)} + \sum_{\ell=3}^{\infty} \frac{\SNR^\ell}{2\ell} = \left( \frac{1}{w_2} - \frac{1}{2}\right) \SNR + \left( \frac{1}{2(w_4-1)} - \frac{1}{4} \right) \SNR^2 - \frac{1}{2} \log(1-\SNR),
\eeq
which proves the first part of the theorem.

Since we only used Cauchy's inequality, the equality in \eqref{eq:Cauchy_ineq} holds if and only if
\beq \label{eq:Cauchy_equal}
	\frac{w_2 \tau_1(f)}{\sqrt{\SNR}} = \frac{2(w_4-1) \tau_2(f)}{\SNR} = \frac{2\ell \tau_{\ell}(f)}{\sqrt{\SNR^{\ell}}} \qquad (\ell = 3, 4, \dots).
\eeq
We now find all functions $f$ that satisfy \eqref{eq:Cauchy_equal}. Letting $2C$ be the common value in \eqref{eq:Cauchy_equal}, we rewrite \eqref{eq:Cauchy_equal} as
\beq \label{eq:tau_ell}
	\tau_1(f) = \frac{2C \sqrt{\SNR}}{w_2}, \quad \tau_2(f) = \frac{C \SNR}{w_4-1}, \quad \tau_{\ell}(f) = \frac{C \sqrt{\SNR^{\ell}}}{\ell} \qquad (\ell = 3, 4, \dots).
\eeq

Since $f$ is analytic, we can consider the Taylor expansion of it. Using the Chebyshev polynomials, we can expand $f$ as
\beq
	f(x) = \sum_{\ell=0}^{\infty} C_{\ell} T_{\ell} \left( \frac{x}{2} \right).
\eeq
Then, from the orthogonality relation of the Chebyshev polynomials, we get for $\ell \geq 1$ that
\beq
	\tau_{\ell}(f) = \frac{C_{\ell} }{\pi} \int_{-2}^2 T_{\ell} \left( \frac{x}{2} \right) T_{\ell} \left( \frac{x}{2} \right) \frac{\dd x}{\sqrt{4-x^2}} = \frac{C_{\ell} }{\pi} \int_{-1}^1 T_{\ell} \left( y \right) T_{\ell} \left( y \right) \frac{\dd y}{\sqrt{1-y^2}} = \frac{C_{\ell}}{2}.
\eeq
Thus, \eqref{eq:tau_ell} holds if and only if
\beq \begin{split} \label{eq:pre_optimized}
	f(x) &= c_0 + 2C \left( \frac{2 \sqrt{\SNR}}{w_2} T_1 \left( \frac{x}{2} \right) + \frac{\SNR}{w_4-1} T_2 \left( \frac{x}{2} \right) + \sum_{\ell=3}^{\infty} \frac{\sqrt{\SNR^{\ell}}}{\ell} T_{\ell} \left( \frac{x}{2} \right) \right) \\
	&= c_0 +2 C \sqrt{\SNR} \left( \frac{2}{w_2} - 1 \right) T_1 \left( \frac{x}{2} \right) + 2C \SNR \left( \frac{1}{w_4-1} - \frac{1}{2} \right)  T_2 \left( \frac{x}{2} \right) + 2C \sum_{\ell=3}^{\infty} \frac{\sqrt{\SNR^{\ell}}}{\ell} T_{\ell} \left( \frac{x}{2} \right)
\end{split} \eeq
for some constant $c_0$. It is well-known from the generating function of the Chebyshev polynomials that
\beq
	\sum_{\ell=1}^{\infty} \frac{t^{\ell}}{\ell} T_{\ell} \left( x \right) = \log \left( \frac{1}{\sqrt{1-2tx+t^2}} \right).
\eeq
(See, e.g., (18.12.9) of \cite{Handbook}.) Since $T_1(x) = x$ and $T_2(x) = 2x^2 -1$, we find that \eqref{eq:pre_optimized} is equivalent to
\beq \label{eq:optimized_phi}
	f(x) = c_0 + C \sqrt{\SNR} \left( \frac{2}{w_2} - 1 \right) x + C \SNR \left( \frac{1}{w_4-1} - \frac{1}{2} \right) \left( x^2 -2 \right) + C \log \left( \frac{1}{1-\sqrt{\SNR} x + \SNR} \right).
\eeq
This concludes the proof of Theorem \ref{thm:optimize}.

\subsection{Proof of Theorem \ref{thm:CLT}} \label{sec:CLT_proof}

We adapt the strategy of Bai and Silverstein \cite{Bai-Silverstein2004}, and Bai and Yao \cite{Bai-Yao2005}. In this method, we first express the left-hand side of \eqref{eq:CLT_LSS} by using a contour integral via Cauchy's integration formula. The integral is then written in terms of the Stieltjes transforms of the empirical spectral measure and the semicircle measure. Since the Stieltjes transform of the empirical spectral measure converges weakly to a Gaussian process, we find that the linear eigenvalue statistic also converges to a Gaussian random variable. Precise control of error terms requires estimates on the resolvents from random matrix theory, which are known as the local laws.

Denote by $\rho_N$ the empirical spectral distribution of $M$, i.e.,
\beq
	\rho_N = \frac{1}{N} \sum_{i=1}^N \delta_{\mu_i}.
\eeq 
As $N \to \infty$, $\rho_N$ converges to the Wigner semicircle measure $\rho$, defined by
\beq
	\rho(\dd x) = \frac{\sqrt{(4-x^2)_+}}{2\pi} \dd x.
\eeq
Choose ($N$-independent) constants $a_- \in (-3, -2)$, $a_+ \in (2, 3)$, and $v_0 \in (0, 1)$ so that the function $f$ is analytic on the rectangular contour $\Gamma$ whose vertices are $(a_- \pm \ii v_0)$ and $(a_+ \pm \ii v_0)$. Since $\| M \| \to 2$ almost surely, we assume that all eigenvalues of $M$ are contained in $\Gamma$. Thus, from Cauchy's integral formula, we find that
\beq \begin{split} \label{eq:decouple1}
	\sum_{i=1}^N f(\mu_i) &= \sum_{i=1}^N \frac{1}{2\pi \ii} \oint_{\Gamma} \frac{f(z)}{z-\mu_i} \dd z = \frac{1}{2\pi \ii} \oint_{\Gamma} f(z) \left( \sum_{i=1}^N \frac{1}{z-\mu_i} \right) \dd z \\
	&= -\frac{N}{2\pi \ii} \oint_{\Gamma} f(z) \left( \int_{-\infty}^{\infty} \frac{\rho_N (\dd x)}{x-z} \right) \dd z.
\end{split} \eeq
The procedure decouples the randomness of $\mu_i$ and the function $f$, and we can solely focus on the randomness of $\mu_i$ via the integral of the function $(x-z)^{-1}$ with respect to the random measure $\rho_N(\dd x)$. 

Let us recall the Stieltjes transform to handle the random integral of $(x-z)^{-1}$. For a measure $\nu$ and a variable $z \in \C^+$, the Stieltjes transform $s_{\nu}(z)$ of $\nu$ is defined by
\beq
	s_{\nu}(z) = \int_{-\infty}^{\infty} \frac{\nu(\dd x)}{x-z}.
\eeq
We abbreviate $s_{\rho_N}(z) \equiv s_N(z)$. Then, \eqref{eq:decouple1} can be rewritten as
\beq
	\sum_{i=1}^N f(\mu_i) = -\frac{N}{2\pi \ii} \oint_{\Gamma} f(z) s_N(z) \dd z.
\eeq
Similarly, we also find that
\beq
	N \int_{-2}^2 \frac{\sqrt{4-x^2}}{2\pi} f(x) \, \dd x = -\frac{N}{2\pi \ii} \oint_{\Gamma} f(z) s(z) \dd z,
\eeq
where we let $s(z) = s_{\rho}(z)$, the Stieltjes transform of the Wigner semicircle measure. Thus, we obtain that
\beq \label{eq:Cauchy_int}
	\sum_{i=1}^N f(\mu_i) - N \int_{-2}^2 \frac{\sqrt{4-x^2}}{2\pi} f(x) \, \dd x = -\frac{N}{2\pi \ii} \oint_{\Gamma} f(z) \big( s_N(z) - s(z) \big) \dd z.
\eeq
We remark that $s(z)$ satisfies
\beq
	s(z) = \frac{1}{2\pi} \int_{-2}^2 \frac{\sqrt{4-x^2}}{x-z} \dd x = \frac{-z+\sqrt{z^2-4}}{2}.
\eeq

We use the results from the random matrix theory to analyze the right-hand side of \eqref{eq:Cauchy_int}.
For $z \in \C^+$, define the resolvent $R(z)$ of $M$ by
\beq
	R(z) = (M-zI)^{-1}.
\eeq
Note that the normalized trace of the resolvent satisfies
\beq
	\frac{1}{N} \Tr R(z) = \frac{1}{N} \sum_{i=1}^N \frac{1}{\mu_i -z} = s_N(z).
\eeq
Let
\beq
	\xi_N(z) = N( s_N(z) - s(z) ) = \sum_{i=1}^N [ R_{ii}(z) - s(z)].
\eeq

As discussed in Section \ref{sec:intro}, Theorem \ref{thm:CLT} was proved in \cite{Baik-Lee2017} for 
\[
	\bsx = \boldsymbol{1} = \frac{1}{\sqrt N}(1, 1, \dots, 1)^T.
\]
We introduce an interpolation between $\bsx$ and $\boldsymbol{1}$ as follows: Since $\bsx, \boldsymbol{1} \in \S^{N-1}$, the $(N-1)$-dimensional unit sphere, we can consider a parametrized curve $\bsy: [0, 1] \to \S^{N-1}$, a segment of the geodesic on $\S^{N-1}$ joining $\bsx$ and $\boldsymbol{1}$ such that $\bsy(0) = \bsx$ and $\bsy(1) = \boldsymbol{1}$. We write
\beq
	\bsy(\theta) = (y_1(\theta), y_2(\theta), \dots, y_N(\theta))^T
\eeq
and also define
\beq
	M_{ij}(\theta) = \sqrt{\SNR} y_i (\theta) y_j (\theta) + H_{ij},	\quad R(\theta, z) = (M(\theta) -zI)^{-1}, \quad \xi_N(\theta, z) = \sum_{i=1}^N [ R_{ii}(\theta, z) - s(z)].
\eeq

Our strategy of the proof is to show that the limiting distribution of $\xi_N(\theta, z)$ does not change with $\theta$. More precisely, we claim that
\beq \label{eq:fundamental}
	\frac{\partial}{\partial\theta} \xi_N(\theta, z) = \caO(N^{-\frac{1}{2}})
\eeq
uniformly on $z \in \Gamma$. Once we prove the claim, we can use the lattice argument to prove Theorem \ref{thm:CLT} as follows: Choose points $z_1, z_2, \dots, z_{16N} \in \Gamma$ so that $|z_i - z_{i+1}| \leq N^{-1}$ for $i=1, 2, \dots, 16N$ (with the convention $z_{16N+1} = z_1$). For each $z_i$, the claim \eqref{eq:fundamental} shows that
\beq
	\xi_N(1, z_i) - \xi_N(0, z_i) = \caO(N^{-\frac{1}{2}}).
\eeq
For any $z \in \Gamma$, if $z_i$ is the nearest lattice point from $z$, then $|z-z_i| \leq N^{-1}$. From the Lipschitz continuity of $\xi_N$, we then find $|\xi_N(\theta, z) - \xi_N(\theta, z_i)| = \caO(N^{-1})$ uniformly on $z$ and $z_i$. Hence,
\beq
	|\xi_N(1, z) - \xi_N(0, z)| \leq |\xi_N(1, z) - \xi_N(1, z_i)| + |\xi_N(1, z_i) - \xi_N(0, z_i)| + |\xi_N(0, z_i) - \xi_N(0, z)| = \caO(N^{-\frac{1}{2}}).
\eeq
Now, integrating over $\Gamma$, we get
\beq \label{eq:lattice}
	\frac{1}{2\pi \ii} \oint_{\Gamma} f(z) \xi_N(1, z) \dd z - \frac{1}{2\pi \ii} \oint_{\Gamma} f(z) \xi_N(0, z) \dd z = \caO(N^{-\frac{1}{2}}).
\eeq
This shows that the limiting distribution of the right-hand side of \eqref{eq:Cauchy_int} does not change even if we change $\bsx$ into $\boldsymbol{1}$. Therefore, we get the desired theorem from Theorem 1.6 and Remark 1.7 of \cite{Baik-Lee2017}.

We now prove the claim \eqref{eq:fundamental}. For the ease of notation, we omit the $z$-dependence in some occasions. Using the formula
\beq
	\frac{\partial R_{jj}(\theta)}{\partial M_{ab}(\theta)} =
		\begin{cases}
			-R_{ja}(\theta) R_{bj}(\theta) -R_{jb}(\theta) R_{aj}(\theta) & \text{ if } a \neq b, \\
			-R_{ja}(\theta) R_{aj}(\theta) & \text{ if } a=b,
		\end{cases}
\eeq
and the fact that $M$ and $R(\theta)$ are symmetric, it is straightforward to check that
\beq \label{eq:theta_derivative}
	\frac{\partial}{\partial \theta} \xi_N(\theta) = \sum_{a \leq b}^N \frac{\partial M_{ab}(\theta)}{\partial \theta} \frac{\partial \xi_N(\theta)}{\partial M_{ab}(\theta)} = -\sqrt{\SNR} \sum_{a, b=1}^N \dot y_a(\theta) y_b(\theta) \sum_{j=1}^N R_{ja}(\theta) R_{bj}(\theta) \,,
\eeq
where we use the notation $\dot y_a \equiv \dot y_a(\theta) = \frac{\dd y_a(\theta)}{\dd \theta}$. 

To estimate the right-hand side of \eqref{eq:theta_derivative}, we first note that
\beq \label{eq:z_deriavtiave1}
	\sum_{a, b=1}^N \dot y_a(\theta) y_b(\theta) \sum_{j=1}^N R_{ja}(\theta) R_{bj}(\theta) = \langle \dot \bsy(\theta), R(\theta)^2 \bsy(\theta) \rangle
\eeq
For the resolvents of the Wigner matrices, we have the following lemma from \cite{Knowles-Yin2013}.

\begin{lem}[Isotropic local law] \label{lem:local_law}
For an $N$-independent constant $\epsilon > 0$, let $\Gamma^{\epsilon}$ be the $\epsilon$-neighborhood of $\Gamma$, i.e.,
\[
	\Gamma^{\epsilon} = \{ z \in \C : \min_{w \in \Gamma} |z-w| \leq \epsilon \}.
\]
Choose $\epsilon$ small so that the distance between $\Gamma^{\epsilon}$ and $[-2, 2]$ is larger than $2\epsilon$, i.e., 
\beq
	\min_{w \in \Gamma^{\epsilon}, x \in [-2, 2]} |x-w| > 2\epsilon.
\eeq
Then, for any deterministic $\bsv, \bsw \in \C^N$ with $\| \bsv \| = \| \bsw \| = 1$, the following estimate holds uniformly on $z \in \Gamma^{\epsilon}$:
\beq \label{eq:iso_spike}
	\left|\langle \bsv, (H-zI)^{-1} \bsw \rangle - s(z) \langle \bsv, \bsw \rangle \right| = \caO(N^{-\frac{1}{2}}).
\eeq
\end{lem}

\begin{proof}[Proof of Lemma \ref{lem:local_law}]
We prove the lemma by using the results in \cite{Knowles-Yin2013}.
If $z = E + \ii \eta \in \Gamma^{\epsilon}$ for some $E \in [a_- -\epsilon, a_+ +\epsilon]$ and $\eta \in [v_0-\epsilon, v_0 +\epsilon]$, we get the estimate from Theorem 2.2 of \cite{Knowles-Yin2013} since the control parameter $\Psi(z)$ in Equation (2.7) of \cite{Knowles-Yin2013} is bounded by
\[
	\Psi(E+\ii v_0) \equiv \sqrt{\frac{\im s(E+\ii \eta)}{N \eta}} + \frac{1}{N \eta} = O(N^{-\frac{1}{2}}).
\]
A similar estimate holds for $z = E - \ii \eta \in \Gamma^{\epsilon}$ with $E \in [a_- -\epsilon, a_+ +\epsilon]$ and $\eta \in [-v_0-\epsilon, -v_0 +\epsilon]$. On the other hand, if $z = E + \ii \eta \in \Gamma$ for $E \in [a_- -\epsilon, a_- +\epsilon] \cup [a_+ -\epsilon, a_+ +\epsilon]$ and $\eta \in (0, v_0 +\epsilon]$, we can check from an elementary calculation that $|\im s(E+\ii \eta)| \leq C\eta$ for some constant $C$ independent of $N$. Thus, the upper bound in Equation (2.10) of \cite{Knowles-Yin2013} becomes
\[
	\sqrt{\frac{\im s(E +\ii \eta)}{N \eta}} = O(N^{-\frac{1}{2}}).
\]
A similar estimate holds for $z = E - \ii \eta \in \Gamma^{\epsilon}$ with $E \in [a_- -\epsilon, a_- +\epsilon] \cup [a_+ -\epsilon, a_+ +\epsilon]$ and $\eta \in (0, v_0 +\epsilon]$. This completes the proof of the lemma.
\end{proof}

To show that the right-hand side of \eqref{eq:z_deriavtiave1} is negligible, we want to use Lemma \ref{lem:local_law}. The main difference between the right-hand side of \eqref{eq:z_deriavtiave1} and the left-hand side of \eqref{eq:iso_spike} is that the former contains the square of the resolvent, and it is not the resolvent of $H$ but of $M(\theta)$. We can overcome the first difficulty by rewriting $R(\theta, z)$ as
\beq
	R(\theta, z)^2 = (M(\theta)-zI)^{-2} = \frac{\partial}{\partial z} (M(\theta)-zI)^{-1} = \frac{\partial}{\partial z} R(\theta, z),
\eeq
which can be checked from the definition of the resolvent. Hence we find that
\beq
	\langle \dot \bsy(\theta), R(\theta, z)^2 \bsy(\theta) \rangle = \frac{\partial}{\partial z} \langle \dot \bsy(\theta), R(\theta, z) \bsy(\theta) \rangle.
\eeq
Later, we will apply Cauchy's integral formula to estimate the derivative in \eqref{eq:z_deriavtiave1} by an integral of the inner product $\langle \dot \bsy(\theta), R(\theta, z) \bsy(\theta) \rangle$.

Next, we obtain an analogue of Lemma \ref{lem:local_law} by using the resolvent expansion. Set $S(z) = (H-zI)^{-1}$. We have from the definition of the resolvents that
\beq \label{eq:resolvent_diff1}
	R(\theta, z)^{-1} - S(z)^{-1} = \sqrt{\lambda} \bsy(\theta) \bsy(\theta)^T,
\eeq
and after multiplying $S(z)$ from the right and $R(\theta, z)$ from the left, we find that
\beq \label{eq:resolvent_diff2}
	S(z) - R(\theta, z) = \sqrt{\lambda} R(\theta, z) \bsy(\theta) \bsy(\theta)^T S(z).
\eeq
Thus,
\beq \begin{split} \label{eq:resolvent_expansion}
	\langle \dot \bsy(\theta), S(z) \bsy(\theta) \rangle &= \langle \dot \bsy(\theta), R(\theta, z) \bsy(\theta) \rangle + \sqrt{\lambda} \langle \dot \bsy(\theta), R(\theta, z) \bsy(\theta) \bsy(\theta)^T S(z) \bsy(\theta) \rangle \\
	&= \langle \dot \bsy(\theta), R(\theta, z) \bsy(\theta) \rangle + \sqrt{\lambda} \langle \dot \bsy(\theta), R(\theta, z) \bsy(\theta) \rangle \langle \bsy(\theta), S(z) \bsy(\theta) \rangle \\
	&= \langle \dot \bsy(\theta), R(\theta, z) \bsy(\theta) \rangle \left( 1 + \sqrt{\lambda} \langle \bsy(\theta), S(z) \bsy(\theta) \rangle \right).
\end{split} \eeq
From the isotropic local law, Lemma \ref{lem:local_law}, we find that
\beq
	\langle \bsy(\theta), S(z) \bsy(\theta) \rangle = s(z) + \caO(N^{-\frac{1}{2}}).
\eeq
Recall that $\| \bsy(\theta) \| = 1$. Then, it is obvious that $\langle \dot \bsy(\theta), \bsy(\theta) \rangle = \frac{1}{2} \frac{\dd}{\dd \theta} \| \bsy(\theta) \|^2 = 0$. Hence, again from Lemma \ref{lem:local_law}, we also find that
\beq
	\langle \dot \bsy(\theta), S(z) \bsy(\theta) \rangle = s(z) \langle \dot \bsy(\theta), \bsy(\theta) \rangle + \caO(N^{-\frac{1}{2}}) = \caO(N^{-\frac{1}{2}}).
\eeq
We then have from \eqref{eq:resolvent_expansion} that
\beq
	\langle \dot \bsy(\theta), R(\theta, z) \bsy(\theta) \rangle = \frac{\langle \dot \bsy(\theta), S(z) \bsy(\theta) \rangle}{1 + \sqrt{\lambda} \langle \bsy(\theta), S(z) \bsy(\theta) \rangle} = \caO(N^{-\frac{1}{2}}),
\eeq
where we used that $|s| \leq 1$ and $\lambda < 1$, hence $|1+\sqrt{\lambda} s| > c > 0$ for some ($N$-independent) constant $c$.

Consider the boundary of the $\epsilon$-neighborhood of $z$, $\partial B_{\epsilon}(z) = \{ w \in \C : |w-z| = \epsilon \}$. If we choose $\epsilon$ as in the assumption of Lemma \ref{lem:local_law}, $\partial B_{\epsilon}(z)$ does not intersect $[-2, 2]$. Applying Cauchy's integral formula, we get
\beq \label{eq:z_deriavtiave2}
	\frac{\partial}{\partial z} \langle \dot \bsy(\theta), R(\theta, z) \bsy(\theta) \rangle = \frac{1}{2\pi \ii} \oint_{\partial B_{\epsilon}(z)} \frac{\langle \dot \bsy(\theta), R(\theta, w) \bsy(\theta) \rangle}{(w-z)^2} \dd w = \caO(N^{-\frac{1}{2}}).
\eeq
Thus, we get from \eqref{eq:z_deriavtiave1} and \eqref{eq:z_deriavtiave2} that
\beq
	\langle \dot \bsy(\theta), R(\theta)^2 \bsy(\theta) \rangle = \caO(N^{-\frac{1}{2}}).
\eeq
Plugging the estimate into the right-hand side of \eqref{eq:theta_derivative}, we get the claim \eqref{eq:fundamental}.

\section{Conclusion and Future Works} \label{sec:conclusion}

In this paper, we proposed a hypothesis test for a signal detection problem in a rank-one spiked Wigner model. Based on the central limit theorem for the linear spectral statistics of the data matrix, we established a test statistic that does not require any prior information on the signal. The test and its error is independent of the noise matrix except the variance of the diagonal entries and the fourth moment of the off-diagonal entries. The error of the proposed test is the lowest among all tests based on the linear spectral statistics, and it also matches the error of the likelihood ratio test if the noise is Gaussian. When the density of the noise is known, we further improve the test by adapting the entrywise transformation introduced in \cite{Perry2018}.

An interesting future research direction is to extend the test to the case with a spike of higher ranks or spiked sample covariance matrices. We believe that it is possible to prove the central limit theorem for the linear statistics in these general models to which our test can be naturally extended. We also hope to generalize our results to the data matrix with non-Wigner noise, where the variances of off-diagonal entries of the noise matrix are not identical, including (sparse) stochastic block models.

\subsection*{Acknowledgements}
We thank anonymous referees for their constructive comments that enabled us to improve the manuscript. J. O. Lee thanks Hong Chang Ji and Ji Hyung Jung for helpful discussions. The work of H. W. Chung was partially supported by National Research Foundation of Korea under grant number 2017R1E1A1A01076340. The work of J. O. Lee was partially supported by the Samsung Science and Technology Foundation project number SSTF-BA1402-04.

\begin{appendix}

\section{Computation of the test statistic} \label{sec:compute}

\begin{lem} \label{lem:L_t}
Let
\beq
	L_\lambda = \sum_{i=1}^N \phi_\lambda(\mu_i) - N \int_{-2}^2 \frac{\sqrt{4-y^2}}{2\pi} \phi_\lambda(y) \, \dd y
\eeq
where $\phi_\lambda$ is defined as in \eqref{eq:phi}.
Then,
\beq \label{eq:L_t proof}
	L_\lambda = - \log \det \left( (1+\lambda)I - \sqrt{\lambda}M \right) + \frac{\lambda N}{2} + \sqrt{\lambda} \left( \frac{2}{w_2} - 1 \right) \Tr M + \lambda \left( \frac{1}{w_4-1} - \frac{1}{2} \right) (\Tr M^2 - N).
\eeq
\end{lem}

\begin{proof}
It is straightforward to see that
\beq
	\sum_{i=1}^N \phi_\lambda(\mu_i) = - \log \det \left( (1+\lambda)I - \sqrt{\lambda}M \right) + \sqrt{\lambda} \left( \frac{2}{w_2} - 1 \right) \Tr M + \lambda \left( \frac{1}{w_4-1} - \frac{1}{2} \right) \Tr M^2.
\eeq
To compute the integral in the definition of $L_\lambda$, we use the formula
\beq
	\int_{-2}^2 \log(z-y) \frac{\sqrt{4-y^2}}{2\pi} \, \dd y = \frac{z}{4} \left( z - \sqrt{z^2 -4} \right) + \log \left( z + \sqrt{z^2 -4} \right) - \log 2 - \frac{1}{2}
\eeq
for $z>2$. See, e.g., Equation (8.5) of \cite{Baik-Lee2016}. Putting $z = (1+\lambda)/\sqrt{\lambda}$, we get
\beq
	\int_{-2}^2 \log \left( \frac{1}{1-\sqrt{\lambda} y + \lambda} \right) \frac{\sqrt{4-y^2}}{2\pi} \, \dd y = - \int_{-2}^2 \left( \log \sqrt{\lambda} + \log \left( \frac{1+\lambda}{\sqrt{\lambda}} -y \right) \right) \frac{\sqrt{4-y^2}}{2\pi} \, \dd y = -\frac{\lambda}{2}.
\eeq
Finally, it is elementary to check that
\beq
	\int_{-2}^2 \frac{y\sqrt{4-y^2}}{2\pi} \, \dd y = 0, \qquad \int_{-2}^2 \frac{y^2\sqrt{4-y^2}}{2\pi} \, \dd y = 1.
\eeq
This proves Equation \eqref{eq:L_t proof}.
\end{proof}

\begin{lem} \label{lem:m_M}
Let
\beq
	m_H(\phi_\lambda) = \frac{1}{4} \left( \phi_\lambda(2) + \phi_\lambda(-2) \right) -\frac{1}{2} \tau_0(\phi_\lambda) + (w_2 -2) \tau_2(\phi_\lambda) + (w_4-3) \tau_4(\phi_\lambda)
\eeq
and
\beq
	m_M(\phi_\lambda) = \frac{1}{4} \left( \phi_\lambda(2) + \phi_\lambda(-2) \right) -\frac{1}{2} \tau_0(\phi_\lambda) + (w_2 -2) \tau_2(\phi_\lambda) + (w_4-3) \tau_4(\phi_\lambda) + \sum_{\ell=1}^{\infty} \sqrt{\lambda^{\ell}} \tau_{\ell}(\phi_\lambda)
\eeq
where $\phi_\lambda$ is defined as in \eqref{eq:phi}.
Then,
\beq
	m_H(\phi_\lambda) = -\frac{1}{2} \log(1-\lambda) + \left(\frac{w_2 -1}{w_4-1} -\frac{1}{2} \right)\lambda + \frac{(w_4 -3) \lambda^2}{4}
\eeq
and
\beq
	m_M(\phi_\lambda) = m_H(\phi_\lambda) - \log(1-\lambda) + \left( \frac{2}{w_2} - 1 \right) \lambda + \left( \frac{1}{w_4-1} - \frac{1}{2} \right) \lambda^2.
\eeq
In particular, $m_H(\phi_\lambda) < m_M(\phi_\lambda)$ if $\lambda \in (0, 1)$.
\end{lem}

\begin{proof}
Recall that $\phi_\lambda$ is the function $f$ in \eqref{eq:optimized_phi} with $C=1$ and $c_0 = (\frac{2}{w_4-1} -1)\lambda$. Thus, from \eqref{eq:tau_ell},
\beq \label{eq:tau_ell_phi}
	\tau_1(\phi_\lambda) = \frac{2 \sqrt{\lambda}}{w_2}, \quad \tau_2(\phi_\lambda) = \frac{\lambda}{w_4-1}, \quad \tau_{\ell}(\phi_\lambda) = \frac{\sqrt{\lambda^{\ell}}}{\ell} \qquad (\ell = 3, 4, \dots).
\eeq
Moreover,
\beq
	\tau_0(\phi_\lambda) = c_0 = \left( \frac{2}{w_4-1} -1 \right) \lambda.
\eeq
Since
\beq \begin{split}
	\phi_\lambda(2) + \phi_\lambda(-2) &= \log \left( \frac{1}{1-2\sqrt{\lambda}+\lambda} \right) + \log \left( \frac{1}{1+2\sqrt{\lambda}+\lambda} \right) + 8\lambda \left( \frac{1}{w_4-1} - \frac{1}{2} \right) \\
	&= -2 \log(1-\lambda) + 8\lambda \left( \frac{1}{w_4-1} - \frac{1}{2} \right),
\end{split} \eeq
we find that
\beq \begin{split}
	m_H(\phi_\lambda) &= -\frac{1}{2} \log(1-\lambda) + 2\lambda \left( \frac{1}{w_4-1} - \frac{1}{2} \right) - \frac{\lambda}{2} \left( \frac{2}{w_4-1} -1 \right) + \frac{(w_2 -2)\lambda}{w_4-1} + \frac{(w_4 -3) \lambda^2}{4} \\
	&= -\frac{1}{2} \log(1-\lambda) + \left(\frac{w_2 -1}{w_4-1} -\frac{1}{2} \right)\lambda + \frac{(w_4 -3) \lambda^2}{4}.
\end{split} \eeq
Moreover, we also find that
\beq \begin{split}
	m_M(\phi_\lambda) &= m_H(\phi_\lambda) + \frac{2 \lambda}{w_2} + \frac{\lambda^2}{w_4-1} + \sum_{\ell=3}^{\infty} \frac{\lambda^{\ell}}{\ell} \\
	&= m_H(\phi_\lambda) + \left( \frac{2}{w_2} - 1 \right) \lambda + \left( \frac{1}{w_4-1} - \frac{1}{2} \right) \lambda^2 + \sum_{\ell=1}^{\infty} \frac{\lambda^{\ell}}{\ell} \\
	&= m_H(\phi_\lambda) - \log(1-\lambda) + \left( \frac{2}{w_2} - 1 \right) \lambda + \left( \frac{1}{w_4-1} - \frac{1}{2} \right) \lambda^2.
\end{split} \eeq
Finally, it is obvious $m_M(\phi_\lambda) > m_H(\phi_\lambda)$ if $\lambda \in (0, 1)$ since $\tau_{\ell}(\phi_\lambda) > 0$ for all $\ell = 1, 2, \dots$.
\end{proof}

\begin{rem} \label{rem:m_M}
For any $\lambda$,
\beq
	V_M(\phi_\lambda) = V_H(\phi_\lambda) = -2 \log(1-\lambda) + \left( \frac{4}{w_2} -2 \right) \lambda + \left( \frac{2}{w_4-1} -1 \right) \lambda^2,
\eeq
which can be easily checked from \eqref{eq:tau_ell_phi}.
\end{rem}

\section{Proof of Theorem \ref{thm:trans_CLT}} \label{sec:entry}

Recall that the normalized off-diagonal entries $\sqrt{N} H_{ij}$ are identically distributed with density $g$ and the normalized diagonal entries $\sqrt{N/w_2} H_{ii}$ are identically distributed with density $g_d$. In Assumption \ref{assump:entry}, we further assumed that the densities $g$ and $g_d$ are smooth, positive everywhere, with subexponential tails, and symmetric (about $0$). We also assumed that
\[
	\| \bsx \|_{\infty} = O(N^{-\phi})
\]
for some $\frac{3}{8} < \phi \leq \frac{1}{2}$.

As discussed in Section \ref{sec:trans}, we consider the entrywise transformation defined by a function
\beq
	h(w) := -\frac{g'(w)}{g(w)}.
\eeq
If $\lambda = 0$, it is immediate to see that for $i\neq j$
\[
	\E[h(\sqrt{N} M_{ij})] = \int_{-\infty}^{\infty} h(w) g(w) \dd w = -\int_{-\infty}^{\infty} g'(w) \dd w = 0.
\]
Further, with $\lambda = 0$, as shown in Proposition 4.2 of \cite{Perry2018},
\beq
	\fh := \E[h(\sqrt{N} M_{ij})^2] = \int_{-\infty}^{\infty} h(w)^2 g(w) \dd w = \int_{-\infty}^{\infty} \frac{g'(w)^2}{g(w)} \dd w \geq 1,
\eeq
where the equality holds if and only if $\sqrt{N} H_{ij}$ is a standard Gaussian (hence $h(w) = w$).
For the diagonal entries, we similarly define
\beq
	h_d(w) := -\frac{g_d'(w)}{g_d(w)}.
\eeq
Then, if $\lambda=0$, $\E[h_d(\sqrt{N/w_2} M_{ii})] = 0$ and
\beq
	\fh_d := \E[h_d(\sqrt{N/w_2} M_{ii})^2] = \int_{-\infty}^{\infty} \frac{g_d'(w)^2}{g_d(w)} \dd w \geq 1,
\eeq

We define a transformed matrix $\tM$ as follows: the off-diagonal terms of $\tM$ are defined by
\beq
	\tM_{ij} = \frac{1}{\sqrt{\fh N}} h(\sqrt{N} M_{ij}) \quad (i \neq j), \qquad \tM_{ii} = \sqrt{\frac{w_2}{\fh_d N}} h_d \big(\sqrt{\frac{N}{w_2}} M_{ii} \big).
\eeq
Note that the entries of $\tM$ are independent up to symmetry. Since $g$ is smooth, $h$ is also smooth and all moments of $\sqrt{N} \tM_{ij}$ are $O(1)$. Thus, applying a high-order Markov inequality, it is immediate to find that $\tM_{ij} = \caO(N^{-\frac{1}{2}})$.

\subsection{Decomposition of the transformed matrix}

We first evaluate the mean and the variance of each off-diagonal entry by using the comparison method with the pre-transformed entries. For $i\neq j$, we find that
\beq \begin{split}
	\E[\tM_{ij}] &= \frac{1}{\sqrt{\fh N}} \int_{-\infty}^{\infty} h(w) g(w - \sqrt{\lambda N} x_i x_j) \dd w \\
	&= -\frac{1}{\sqrt{\fh N}} \int_{-\infty}^{\infty} \frac{g'(w)}{g(w)} \left( g(w - \sqrt{\lambda N} x_i x_j) - g(w) \right) \dd w.
\end{split} \eeq
In the Taylor expansion
\beq \begin{split}
	&g(w - \sqrt{\lambda N} x_i x_j) - g(w) \\
	&= \sum_{\ell=1}^4 \frac{g^{(\ell)}(w)}{\ell!} \left( -\sqrt{\lambda N} x_i x_j \right)^{\ell} + \frac{g^{(5)}(w - \theta \sqrt{\lambda N} x_i x_j)}{5!} \left( -\sqrt{\lambda N} x_i x_j \right)^5 
\end{split} \eeq
for some $\theta \in (0, 1)$. Note that the second term and the fourth term in the summation are even functions. Since $g'/g$ is an odd function, from the symmetry we find that
\beq \begin{split} \label{eq:tM_mean}
	\E[\tM_{ij}] &= \frac{\sqrt{\lambda} x_i x_j}{\sqrt{\fh}} \int_{-\infty}^{\infty} \frac{g'(w)^2}{g(w)} \dd w + C_3 N x_i^3 x_j^3 + O(N^2 x_i^5 x_j^5) \\
	&= \sqrt{\lambda\fh} x_i x_j + C_3 N x_i^3 x_j^3 + O(N^2 x_i^5 x_j^5)
\end{split} \eeq
for some ($N$-independent) constant $C_3$.
Similarly,
\beq \begin{split} \label{eq:tM_second}
	\E[\tM_{ij}^2] &= \frac{1}{\fh N} \int_{-\infty}^{\infty} \left( \frac{g'(w)}{g(w)} \right)^2 g(w - \sqrt{\lambda N} x_i x_j) \dd w \\
	&= \frac{1}{N} + \frac{1}{\fh N} \int_{-\infty}^{\infty} \left( \frac{g'(w)}{g(w)} \right)^2 \left( g(w - \sqrt{\lambda N} x_i x_j) - g(w) \right) \dd w \\
	&= \frac{1}{N} + \frac{\lambda x_i^2 x_j^2}{2\fh} \int_{-\infty}^{\infty} \frac{g'(w)^2 g''(w)}{g(w)^2} \dd w + O(N x_i^4 x_j^4) = \frac{1}{N} + \lambda \gh x_i^2 x_j^2 + O(N x_i^4 x_j^4).
\end{split} \eeq

For the diagonal entries, we similarly get
\beq \label{eq:tM_diag_1}
	\E[\tM_{ii}] = \sqrt{\lambda \fh_d} x_i^2 + O(N x_i^6)
\eeq
and
\beq \begin{split} \label{eq:tM_diag_2}
	\E[\tM_{ii}^2] = \frac{w_2}{N} + \frac{\lambda x_i^4}{2\fh} \int_{-\infty}^{\infty} \frac{g_d'(w)^2 g_d''(w)}{g_d(w)^2} \dd w + O(N x_i^8) =: \frac{w_2}{N} + \lambda \gh_d x_i^4 + O(N x_i^8).
\end{split} \eeq
We omit the detail.

The evaluation of the mean and the variance shows that the transformed matrix $\tM$ is not a spiked Wigner matrix when $\lambda > 0$, since the variances of the off-diagonal entries are not identical. Our strategy is to approximate $\tM$ as a spiked generalized Wigner matrix for which the sum of the variances of the entries in each row is equal to $1$. 
Let $S$ be the variance matrix of $\tM$ defined as
\beq
	S_{ij} = \E[\tM_{ij}^2] - (\E[\tM_{ij}])^2.
\eeq
From \eqref{eq:tM_mean}, \eqref{eq:tM_second}, \eqref{eq:tM_diag_1}, and \eqref{eq:tM_diag_2},
\beq
	S_{ij} = \frac{1}{N} + \lambda (\gh - \fh) x_i^2 x_j^2 + O(N \| \bsx \|_{\infty}^{8}) \quad (i \neq j), \quad S_{ii} = \frac{w_2}{N} + \lambda (\gh_d - \fh_d) x_i^4 + O(N \| \bsx \|_{\infty}^{8}),
\eeq
hence
\beq \begin{split}
	\sum_{j=1}^N S_{ij} &= \frac{w_2}{N} + \lambda (\gh_d - \fh_d) x_i^4 + \sum_{j:j \neq i} \left( \frac{1}{N} + \lambda (\gh - \fh) x_i^2 x_j^2 \right) + O(N^2 \| \bsx \|_{\infty}^{8}) \\
	&= 1 + \frac{w_2 - 1}{N} + \lambda (\gh - \fh) x_i^2 + O(N^2 \| \bsx \|_{\infty}^{8}),
\end{split} \eeq
which shows that $\tM$ is indeed approximately a spiked generalized Wigner matrix.

\subsection{CLT for a general Wigner-type matrix}

To adapt the strategy of Section \ref{sec:CLT_proof}, we use the local law for general Wigner-type matrices in \cite{OskariErdosKruger}. Consider a matrix $W = (W_{ij})_{1 \leq i, j \leq N}$ defined by
\beq
	W_{ij} = \frac{1}{\sqrt{N S_{ij}}} (\tM_{ij} - \E[\tM_{ij}]) \quad (i \neq j), \qquad W_{ii} = \sqrt{\frac{w_2}{N S_{ii}}} (\tM_{ii} - \E[\tM_{ii}])
\eeq
Note that $\E[W_{ij}] = 0$, $\E[W_{ij}^2] = \frac{1}{N}$ $(i \neq j)$, and $\E[W_{ii}^2] = \frac{w_2}{N}$.
Then, the matrix $W$ is a Wigner matrix. We set 
\beq
	R^W(z) = (W-zI)^{-1} \quad (z \in \C^+).
\eeq

Next, we introduce an interpolation for $W$. For $0 \leq \theta \leq 1$, we define a matrix $W(\theta)$ by
\beq \begin{split} \label{eq:def_W_theta}
	W_{ij}(\theta) &= (1-\theta) W_{ij} + \theta (\tM_{ij} - \E[\tM_{ij}]) = \left( 1-\theta + \theta \sqrt{N S_{ij}} \right) W_{ij} \\
	&= \left( 1+ \frac{\theta N \lambda (\gh - \fh) x_i^2 x_j^2}{2} + O(N^2 x_i^4 x_j^4) \right) W_{ij} \quad (i \neq j)
\end{split} \eeq
and
\beq \begin{split}
	W_{ii}(\theta) &= (1-\theta) W_{ii} + \theta (\tM_{ii} - \E[\tM_{ii}]) = \left( 1-\theta + \theta \sqrt{\frac{N S_{ii}}{w_2}} \right) W_{ii} \\
	&= \left( 1+ \frac{\theta N \lambda (\gh_d - \fh_d) x_i^4}{2w_2} + O(N^2 x_i^8) \right) W_{ii}.
\end{split} \eeq
Note that $W(0) = W$ and $W(1) = \tM - \E[\tM]$. For $0 \leq \theta \leq 1$, $W(\theta)$ is a general Wigner-type matrix considered in \cite{OskariErdosKruger} satisfying the conditions (A)-(D) therein. Moreover, if we let
\beq \label{eq:def_R^W}
	R^W(\theta, z) = (W(\theta)-zI)^{-1} \quad (z \in \C^+),
\eeq
then Theorem 1.7 of \cite{OskariErdosKruger} asserts that the limiting distribution of $R^W_{ij}(z)$ is $m_i(z) \delta_{ij}$, where $m_i(\theta, z)$ is the unique solution to the quadratic vector equation
\beq
	-\frac{1}{m_i(\theta, z)} = z + \sum_{j=1}^N \E[W_{ij}(\theta)^2] m_j (\theta, z).
\eeq 
Recall that $s(z) = (-z+\sqrt{z^2-4})/2$ is the Stieltjes transform of the Wigner semicircle measure. It is direct to check that $1+zs(z)+s(z)^2=0$. With an ansatz $m_i(\theta, z) = s(z) + C_1 x_i^2 + C_2 N^{-1}$, we can then find $m_i(\theta, z) = s(z) + O(\| \bsx \|_{\infty}^{2})$; see also Theorem 4.2 of \cite{OskariErdosKruger}.

For the resolvent $R^W(\theta, z)$, we have the following lemma from \cite{OskariErdosKruger}.

\begin{lem}[Anisotropic local law] \label{lem:anisotropic}
Let $\Gamma^{\epsilon}$ be the $\epsilon$-neighborhood of $\Gamma$ as in Lemma \ref{lem:local_law}. Then, for any deterministic $\bsv = (v_1, \dots, v_N), \bsw = (w_1, \dots, w_N) \in \C^N$ with $\| \bsv \| = \| \bsw \| = 1$, the following estimate holds uniformly on $z \in \Gamma^{\epsilon} \cap \{ z \in \C^+ : \im z > N^{-\frac{1}{2}} \}$:
\beq \label{eq:aniso_spike}
	\left|\sum_{i, j=1}^N \overline{v_i} R^W_{ij}(\theta, z) w_j - \sum_{i=1}^N m_i(\theta, z) \overline{v_i} w_i \right| = \caO(N^{-\frac{1}{2}}).
\eeq
\end{lem}

\begin{proof}
See Theorem 1.13 of \cite{OskariErdosKruger}. Note that $\rho(z), \kappa(z) = O(\im z)$ in Theorem 1.13 of \cite{OskariErdosKruger}, which can be checked from Equations (1.25), (4.5a), (4.5f), and (1.17) of \cite{OskariErdosKruger}.
\end{proof}

Let $\Gamma_{1/2} := \Gamma \cap \{ z \in \C^+ : |\im z| > N^{-\frac{1}{2}} \}$. On $\Gamma_{1/2}$, as a simple corollary to Lemma \ref{lem:anisotropic}, we obtain
\beq \label{eq:iso_general}
	\left| \langle \bsv, R^W(\theta, z), \bsw \rangle - s(z) \langle \bsv, \bsw \rangle \right| = \caO(N^{-\frac{1}{2}}),
\eeq
which is analogous to Lemma \ref{lem:local_law}. 

We have the following lemma for the difference between $\Tr R^W(0, z)$ and $\Tr R^W(1, z)$ on $\Gamma_{1/2}$.

\begin{lem} \label{lem:Tr_R^W}
Let $R^W(\theta, z)$ be defined as in Equations \eqref{eq:def_W_theta} and \eqref{eq:def_R^W}. Then, the following holds uniformly for $z \in \Gamma_{1/2}$:
\beq \label{eq:W_inter_claim}
	\Tr R^W(1, z) - \Tr R^W(0, z) = \lambda (\gh - \fh) s'(z) s(z) + \caO(N^{\frac{3}{2}} \| \bsx \|_{\infty}^{4}).
\eeq
\end{lem}

We will prove Lemma \ref{lem:Tr_R^W} later in this section.

On $\Gamma \backslash \Gamma_{1/2}$, we use the following results on the rigidity of eigenvalues.
\begin{lem} \label{lem:rigidity}
Denote by $\mu^W_1 (\theta) \geq \mu^W_2 (\theta) \geq \dots \geq \mu^W_N (\theta)$ the eigenvalues of $W(\theta)$. Let $\gamma_i$ be the classical location of the eigenvalues with respect to the semicircle measure defined by
\beq
	\int_{\gamma_i}^2 \frac{\sqrt{4-x^2}}{2\pi} \dd x = \frac{1}{N} \left( i - \frac{1}{2} \right)
\eeq
for $i=1, 2, \dots, N$. Then,
\beq
	|\mu^W_i (\theta) - \gamma_i| =\caO(N^{-\frac{2}{3}}).
\eeq
\end{lem}

\begin{proof}
See Corollary 1.11 of \cite{OskariErdosKruger}. Note that it is $O(N^{-\frac{3}{4}})$, hence negligible, the difference between the semicircle measure and the limiting measure $\rho$ in Equation (1.31) and Corollary 1.11 of \cite{OskariErdosKruger}, since $m_i(\theta, z) = s(z) + O(\| \bsx \|_{\infty}^{2})$ for all $i$.
\end{proof}

From Lemma \ref{lem:rigidity}, we find that
\beq \begin{split} \label{eq:W_edge}
	|\Tr R^W(1, z) - \Tr R^W(0, z)| &= \left| \sum_{i=1}^N \left( \frac{1}{\mu^W_i(1) - z} - \frac{1}{\mu^W_i(0) - z} \right) \right| = \left| \sum_{i=1}^N \frac{\mu^W_i(0) - \mu^W_i(1)}{(\mu^W_i(1) - z)(\mu^W_i(0) - z)} \right| \\
	&\leq \left| \sum_{i=1}^N \frac{|\mu^W_i(0) - \gamma_i| + |\gamma_i - \mu^W_i(1)|}{(\mu^W_i(1) - z)(\mu^W_i(0) - z)} \right| = \caO(N^{\frac{1}{3}})
\end{split} \eeq
uniformly for $z \in \Gamma$.
Thus, from \eqref{eq:W_inter_claim} and \eqref{eq:W_edge},
\beq \begin{split} \label{eq:tM_chain_1}
	&\frac{1}{2\pi \ii} \oint_{\Gamma} f(z) \Tr R^W(1, z) \dd z - \frac{1}{2\pi \ii} \oint_{\Gamma} f(z) \Tr R^W(0, z) \dd z \\
	&= \frac{1}{2\pi \ii} \int_{\Gamma_{1/2}} f(z) \left( \Tr R^W(1, z) - \Tr R^W(0, z) \right) \dd z + \frac{1}{2\pi \ii} \int_{\Gamma \backslash \Gamma_{1/2}} f(z) \left( \Tr R^W(1, z) - \Tr R^W(0, z) \right) \dd z \\
	&= \frac{\lambda (\gh - \fh)}{2\pi \ii} \int_{\Gamma_{1/2}} f(z) s'(z) s(z) \dd z + \caO(N^{\frac{3}{2}} \| \bsx \|_{\infty}^{4}) + \caO(N^{-\frac{1}{6}}) \\
	&= \frac{\lambda (\gh - \fh)}{2\pi \ii} \oint_{\Gamma} f(z) s'(z) s(z) \dd z + \caO(N^{\frac{3}{2}} \| \bsx \|_{\infty}^{4}) + \caO(N^{-\frac{1}{6}})
\end{split} \eeq

\subsection{CLT for a general Wigner-type matrix with a spike}

Recall that $W(1) = \tM - \E[\tM]$. 
Our next step in the approximation is to consider $\tM = W(1) + \E[\tM]$. Since $\E[\tM]$ is not a rank-$1$ matrix, we instead consider
\beq
	A(\theta) = W(1) + \theta \sqrt{\lambda\fh} \bsx \bsx^T, \qquad R^A(\theta, z) = (A(\theta)-zI)^{-1}
\eeq
for $\theta \in [0, 1]$. Note that $A(0) = W(1)$. 

We follow the same strategy as in Section \ref{sec:CLT_proof}. For $z \in \Gamma_{1/2}$, we use
\beq \begin{split} \label{eq:partial_trace_R^A}
	\frac{\partial}{\partial \theta} \Tr R^A(\theta, z) &= -\sum_{i=1}^N \sum_{a, b=1}^N \frac{\partial A_{ab}(\theta)}{\partial \theta} R^A_{ia}(\theta, z) R^A_{bi}(\theta, z) \\
	&= -\sqrt{\lambda\fh} \frac{\partial}{\partial z} \sum_{a, b=1}^N x_a x_b R^A_{ba}(\theta, z) = -\sqrt{\lambda\fh} \frac{\partial}{\partial z} \langle \bsx, R^A(\theta, z) \bsx \rangle.
\end{split} \eeq
Recall that $R^A(0, z) = R^W(1, z)$ satisfies the isotropic local law in \eqref{eq:iso_general},
\beq
	\left| \langle \bsv, R^A(0, z), \bsw \rangle - s(z) \langle \bsv, \bsw \rangle \right| = \left| \langle \bsv, R^W(1, z), \bsw \rangle - s(z) \langle \bsv, \bsw \rangle \right| = \caO(N^{-\frac{1}{2}}).
\eeq
As in \eqref{eq:resolvent_diff1} and \eqref{eq:resolvent_diff2}, we can easily check that
\beq \label{eq:local_law_R^A}
	R^A(0, z) - R^A(\theta, z) = \theta \sqrt{\lambda\fh} R^A(\theta, z) \bsx \bsx^T R^A(0, z),
\eeq
hence
\beq
	\langle \bsx, R^A(0, z) \bsx \rangle = \langle \bsx, R^A(\theta, z) \bsx \rangle + \theta \sqrt{\lambda\fh} \langle \bsx, R^A(\theta, z) \bsx \rangle \langle \bsx, R^A(0, z) \bsx \rangle.
\eeq
We thus find that
\beq
	\langle \bsx, R^A(\theta, z) \bsx \rangle = \frac{\langle \bsx, R^A(0, z) \bsx \rangle}{1+\theta \sqrt{\lambda\fh} \langle \bsx, R^A(0, z) \bsx \rangle} = \frac{s(z)}{1+\theta \sqrt{\lambda\fh} s(z)} + \caO(N^{-\frac{1}{2}}).
\eeq
Plugging it back to \eqref{eq:partial_trace_R^A} and applying Cauchy's integral formula again, we find that
\beq
	\frac{\partial}{\partial \theta} \Tr R^A(\theta, z) = -\frac{\sqrt{\lambda\fh} s'(z)}{(1+\theta \sqrt{\lambda\fh} s(z))^2} + \caO(N^{-\frac{1}{2}}).
\eeq
Now, integrating over $\theta$, we get
\beq \begin{split} \label{eq:spike_main}
	\Tr R^A(1, z) - \Tr R^A(0, z) &= \frac{s'(z)}{s(z)} \left(\frac{1}{1+\theta \sqrt{\lambda\fh} s(z)} \right) \Bigg|_{\theta=0}^{\theta=1} + \caO(N^{-\frac{1}{2}}) \\
	&= -\frac{\sqrt{\lambda\fh} s'(z)}{1+\sqrt{\lambda\fh} s(z)} + \caO(N^{-\frac{1}{2}}).
\end{split} \eeq

On $\Gamma \backslash \Gamma_{1/2}$, we use the interlacing property of the eigenvalues. Let $E^A_0$ and $E^A_1$ be the cumulative distribution functions for the eigenvalue counting measures of $A(0)$ and $A(1)$, respectively, i.e., if we let $\mu^A_i(\theta)$ be the $i$-th eigenvalue of $A(\theta)$ and denote by $\mu^A_1(\theta) \geq \mu^A_2(\theta) \geq \dots \geq \mu^A_N(\theta)$ the eigenvalues of $A(\theta)$, then
\beq
	E^A_0(w) = \frac{1}{N} | \{ \mu^A_i(0) : \mu^A_i(0) < w \}|, \qquad E^A_1(w) = \frac{1}{N} | \{ \mu^A_i(1) : \mu^A_i(1) < w \}|.
\eeq
The interlacing property is that
\beq
	N|E^A_0(w) - E^A_1(w)| \leq 1.
\eeq
In terms of $E^A_0$, we can represent the trace of the resolvent $R^A(0, z)$ by
\beq
	\Tr R^A(0, z) = \sum_{i=1}^N \frac{1}{\mu^A_i(0) -z} = N\int_{-\infty}^{\infty} \frac{E^A_0(x)}{(x-z)^2} \dd x,
\eeq
where we used integration by parts with empirical spectral measure of $A(0)$. Similarly,
\[
	\Tr R^A(1, z) = N\int_{-\infty}^{\infty} \frac{E^A_1(x)}{(x-z)^2} \dd x,
\]
and we get
\beq
	\Tr R^A(1, z) - \Tr R^A(0, z) = N\int_{-\infty}^{\infty} \frac{E^A_1(x) - E^A_0(x)}{(x-z)^2} \dd x.
\eeq
From the rigidity, Lemma \ref{lem:rigidity}, we have that $\| A(0) \| -2 = o(1)$. With $A(0) = W(1)$, we decompose $A(\theta)$ into
\[
	A(\theta) = A(0) + (A(\theta) - A(0)) = W(0) + (A(\theta) - A(0)) + W(1) - W(0).
\]
We then find that $\| W(0) + (A(\theta) - A(0)) \| - 2 = o(1)$ with high probability since $A(\theta)$ is a rank-$1$ perturbation of $A(0)$ with $\| A(0) - A(\theta) \| < 1$. It is now not hard to see that $\| A(\theta) \| -2 = o(1)$ with high probability as well, since $\| W(1) - W(0) \| = o(1)$ with high probability (see, e.g., Lemma \ref{lem:rigidity}). Thus,
\beq \label{eq:spike_edge}
	\Tr R^A(1, z) - \Tr R^A(0, z) = N\int_{-\infty}^{\infty} \frac{E^A_1(x) - E^A_0(x)}{(x-z)^2} \dd x = N\int_{-2-\epsilon}^{2+\epsilon} \frac{E^A_1(x) - E^A_0(x)}{(x-z)^2} \dd x = \caO(1).
\eeq

Following the idea in \eqref{eq:tM_chain_1}, we obtain from \eqref{eq:spike_main} and \eqref{eq:spike_edge} that
\beq \begin{split} \label{eq:tM_chain_2}
	&\frac{1}{2\pi \ii} \oint_{\Gamma} f(z) \Tr R^A(1, z) \dd z - \frac{1}{2\pi \ii} \oint_{\Gamma} f(z) \Tr R^A(0, z) \dd z \\
	&= -\frac{1}{2\pi \ii} \oint_{\Gamma} f(z) \frac{\sqrt{\lambda\fh} s'(z)}{1+\sqrt{\lambda\fh} s(z)} \dd z + \caO(N^{-\frac{1}{2}}).
\end{split} \eeq

\subsection{CLT for a general Wigner-type matrix with a spike and small perturbation}

While the rank-$1$ spike in $A$ is $\sqrt{\lambda\fh} \bsx \bsx^T$, the mean of the diagonal entry
\beq
	\E[\tM_{ii}] = \sqrt{\lambda \fh_d} x_i^2 + O(N \| \bsx \|_{\infty}^{6}),
\eeq
which is different from $\sqrt{\lambda\fh} x_i^2$ in general. We thus define a matrix $B(\theta)$ for $0 \leq \theta \leq 1$ by
\beq
	B_{ij}(\theta) = A_{ij}(1) \quad (i \neq j), \qquad B_{ii}(\theta) = A_{ii}(1) + \theta( \E[\tM_{ii}] - \sqrt{\lambda \fh} x_i^2 - C_3 N x_i^6)
\eeq
for the constant $C_3$ in \eqref{eq:tM_mean}. By definition, $B(0) = A(1)$ and
\beq
	\tM_{ii} = B_{ii}(1) + C_3 N x_i^6.
\eeq
We also set
\[
	R^B(\theta, z) = (B(\theta)-zI)^{-1}.
\]

For $z \in \Gamma_{1/2}$,
\beq \begin{split} \label{eq:R^B_derive}
	\frac{\partial}{\partial \theta} \Tr R^B(\theta, z) &= -\sum_{i, a=1}^N \left( \E[\tM_{aa}] - \sqrt{\lambda \fh} x_a^2 - C_3 N x_a^6 \right) R^B_{ia}(\theta, z) R^B_{ai}(\theta, z) \\
	&= -\frac{\partial}{\partial z} \sum_{a=1}^N \left( \E[\tM_{aa}] - \sqrt{\lambda \fh} x_a^2 - C_3 N x_a^6 \right) R^B_{aa}(\theta, z).
\end{split} \eeq
Since $\| B(\theta) -A(1) \| = O(\| \bsx \|_{\infty}^{2})$, we find that 
\[
	R^B_{aa}(\theta, z) - R^B_{aa}(0, z) = R^B_{aa}(\theta, z) - R^A_{aa}(1, z) =  O(\| \bsx \|_{\infty}^{2})
\]
for $a=1, 2, \dots, N$. Denote by $\bse_a$ a standard basis vector whose $a$-th coordinate is $1$ and all other coordinates are zero. From \eqref{eq:local_law_R^A}, we find that
\beq
	\langle \bse_a, R^A(0, z) \bsx \rangle = \langle \bse_a, R^A(1, z) \bsx \rangle + \sqrt{\lambda \fh} \langle \bse_a, R^A(1, z) \bsx \rangle \langle \bsx, R^A(0, z) \bsx \rangle,
\eeq
hence
\beq
	\langle \bse_a, R^A(1, z) \bsx \rangle = \frac{\langle \bse_a, \bsx \rangle s(z)}{1+ \sqrt{\lambda \fh} s(z)} + \caO(N^{-\frac{1}{2}}).
\eeq
Using the same argument again, we obtain that
\beq
	R^A_{aa}(1, z) = \langle \bse_a, R^A(1, z) \bse_a \rangle = s(z) - \frac{\sqrt{\lambda \fh} s(z)^2}{1+ \sqrt{\lambda \fh} s(z)} |\langle \bsx, \bse_a \rangle|^2 + \caO(N^{-\frac{1}{2}}) = s(z) + \caO(N^{-\frac{1}{2}}),
\eeq
hence
\beq
	R^B_{aa}(\theta, z) = R^A_{aa}(1, z) + \caO(N^{-\frac{1}{2}}) = s(z) + \caO(N^{-\frac{1}{2}})
\eeq
as well. Thus,
\beq \begin{split}
	&\sum_{a=1}^N \left( \E[\tM_{aa}] - \sqrt{\lambda \fh} x_a^2 - C_3 N x_a^6 \right) R^B_{aa}(\theta, z) \\
	&= \sum_{a=1}^N \left( \E[\tM_{aa}] - \sqrt{\lambda \fh} x_a^2 \right) s(z) + O(N \| \bsx \|_{\infty}^{4}) + \caO(N^{-\frac{1}{2}}) \\
	&= \sqrt{\lambda} (\sqrt{\fh_d} - \sqrt{\fh}) s(z) + \caO(N \| \bsx \|_{\infty}^{4}) + \caO(N^{-\frac{1}{2}})
\end{split} \eeq
and
\beq
	\frac{\partial}{\partial \theta} \Tr R^B(\theta, z) = -\sqrt{\lambda} (\sqrt{\fh_d} - \sqrt{\fh}) s'(z) + \caO(N \| \bsx \|_{\infty}^{4}) + \caO(N^{-\frac{1}{2}}).
\eeq

Applying the estimate $R^B_{aa}(\theta, z) - R^A_{aa}(1, z) =  O(\| \bsx \|_{\infty}^{2})$ on $\Gamma \backslash \Gamma_{1/2}$, we obtain that
\beq \begin{split} \label{eq:tM_chain_3}
	&\frac{1}{2\pi \ii} \oint_{\Gamma} f(z) \Tr R^B(1, z) \dd z - \frac{1}{2\pi \ii} \oint_{\Gamma} f(z) \Tr R^B(0, z) \dd z \\
	&= -\frac{\sqrt{\lambda} (\sqrt{\fh_d} - \sqrt{\fh})}{2\pi \ii} \oint_{\Gamma} f(z) s'(z) \dd z + \caO(\sqrt{N} \| \bsx \|_{\infty}^{2}) + \caO(N^{-\frac{1}{2}}).
\end{split} \eeq

By construction, for all $i, j$,
\beq
	\tM_{ij} = B_{ij}(1) + C_3 N x_i^3 x_j^3 + O(N^2 x_i^5 x_j^5).
\eeq
Set $\bsx^3 = (x_1^3, x_2^3, \dots, x_N^3)^T$, $B' = B(1) + C_3 N \bsx^3 (\bsx^3)^T$, and $R^{B'}(z) = (B'-zI)^{-1}$. Then, $z \in \Gamma_{1/2}$,
\beq
	\langle \bse_a, R^B(z) \bse_a \rangle - \langle \bse_a, R^{B'}(z) \bse_a \rangle = C_3 N \langle \bse_a, R^{B'} \bsx^3 \rangle \langle \bsx^3, R^B \bse_a \rangle = \caO(N \| \bsx \|_{\infty}^{6}).
\eeq
On $\Gamma \backslash \Gamma_{1/2}$, we use the estimate 
\beq
	R^B_{aa}(z) - R^{B'}_{aa}(z) =  O(N \| \bsx \|_{\infty}^{6}).
\eeq
Then,
\beq \begin{split} \label{eq:tM_chain_4}
	\frac{1}{2\pi \ii} \oint_{\Gamma} f(z) \Tr R^{B'}(z) \dd z - \frac{1}{2\pi \ii} \oint_{\Gamma} f(z) \Tr R^B(1, z) \dd z = \caO(N^2 \| \bsx \|_{\infty}^{6}) + \caO(N\sqrt{N} \| \bsx \|_{\infty}^{6}).
\end{split} \eeq
Finally, if we set $E = \tM - B''$, then $E_{ij} = O(N^2 x_i^5 x_j^5)$.
Then, since $\| \bsx \|_{\infty} = N^{-\phi}$ for some $\phi > \frac{3}{8}$,
\beq
	\| E \| \leq \| E \|_{HS} = \left(\sum_{i, j=1}^N |E_{ij}|^2 \right)^{\frac{1}{2}} = O\left( N^2 \| \bsx \|_{\infty}^{8} \left( \sum_{i, j=1}^N x_i^2 x_j^2 \right)^{\frac{1}{2}} \right) = O\left( N^2 \| \bsx \|_{\infty}^{8} \right) = o(N^{-1}).
\eeq
Thus, if we let $R^{\tM}(z) = (\tM -z)^{-1}$, for any $z \in \Gamma_{\epsilon}$,
\beq \label{eq:tM_chain_5}
	\frac{1}{2\pi \ii} \oint_{\Gamma} f(z) \Tr R^{\tM}(z) \dd z - \frac{1}{2\pi \ii} \oint_{\Gamma} f(z) \Tr R^{B'}(z) \dd z = o(1)
\eeq
with high probability.

\subsection{Proof of Theorem \ref{thm:trans_CLT} and Theorem \ref{thm:trans_optimize}}

We are now ready to prove Theorem \ref{thm:trans_CLT}.

Denote by $\wt\mu_1 \geq \wt\mu_2 \geq \dots \geq \wt\mu_N$ the eigenvalues of $\tM$. Recall that we denoted by $\mu^W_1 (0) \geq \mu^W_2 (0) \geq \dots \geq \mu^W_N (0)$ the eigenvalues of $W(0)$. From Cauchy's integral formula, as in \eqref{eq:decouple1}, we have
\beq \begin{split} \label{eq:tM_chain_final}
	&\sum_{i=1}^N f(\wt\mu_i) - N \int_{-2}^2 \frac{\sqrt{4-x^2}}{2\pi} f(x) \, \dd x \\
	&=\left( \sum_{i=1}^N f(\mu^W_i (0)) - N \int_{-2}^2 \frac{\sqrt{4-x^2}}{2\pi} f(x) \, \dd x \right) + \left( \sum_{i=1}^N f(\wt\mu_i) - \sum_{i=1}^N f(\mu^W_i (0))\right) \\
	&=\left( \sum_{i=1}^N f(\mu^W_i (0)) - N \int_{-2}^2 \frac{\sqrt{4-x^2}}{2\pi} f(x) \, \dd x \right) - \left( \frac{1}{2\pi \ii} \oint_{\Gamma} f(z) \Tr R^{\tM}(z) \dd z - \frac{1}{2\pi \ii} \oint_{\Gamma} f(z) \Tr R^W(0, z) \dd z \right).
\end{split} \eeq
Since $W$ is a Wigner matrix, the first term in the right-hand side converges to a Gaussian random variable. To determine the mean and the variance of the limiting Gaussian distribution, we need to find the leading order term in the fourth moment of $W_{ij}$. We compute as in \eqref{eq:tM_second} to obtain
\beq \begin{split} \label{eq:tM_fourth}
	\E[\tM_{ij}^4] &= \frac{1}{(\fh N)^2} \int_{-\infty}^{\infty} \left( \frac{g'(w)}{g(w)} \right)^4 g(w - \sqrt{\lambda N} x_i x_j) \dd w \\
	&=: \frac{\wt{w_4}}{N^2} + \frac{1}{\fh N} \int_{-\infty}^{\infty} \left( \frac{g'(w)}{g(w)} \right)^4 \left( g(w - \sqrt{\lambda N} x_i x_j) - g(w) \right) \dd w,
\end{split} \eeq
where the first term is the leading term of $\E[\tM_{ij}^4]$ and hence the leading term of $\E[W_{ij}^4]$ as well. The difference between $\E[W_{ij}^4]$ and $\wt{w_4}$ is negligible in the sense that it has no contribitution in the limiting behavior of the resolvent, which can be checked from standard Green function comparison theorems. (See, e.g., Theorem 2.3 of \cite{Erdos-Yau-Yin2011b}.)
Thus, the mean and the variance of the limiting Gaussian distribution are given by
\beq
	m_W(f) = \frac{1}{4} \left( f(2) + f(-2) \right) -\frac{1}{2} \tau_0(f) + (w_2 -2) \tau_2(f) + (\wt{w_4}-3) \tau_4(f)
\eeq
and
\beq
	V_W(f) = (w_2-2) \tau_1(f)^2 + 2(\wt{w_4}-3) \tau_2(f)^2 + 2\sum_{\ell=1}^{\infty} \ell \tau_{\ell}(f)^2,
\eeq
respectively.

For the second term in the right-hand side of \eqref{eq:tM_chain_final}, combining \eqref{eq:tM_chain_1}, \eqref{eq:tM_chain_2}, \eqref{eq:tM_chain_3}, \eqref{eq:tM_chain_4}, and \eqref{eq:tM_chain_5}, we obtain that
\beq \begin{split}
	&\frac{1}{2\pi \ii} \oint_{\Gamma} f(z) \Tr R^{\tM}(z) \dd z - \frac{1}{2\pi \ii} \oint_{\Gamma} f(z) \Tr R^W(0, z) \dd z \\
	&=\frac{\lambda (\gh - \fh)}{2\pi \ii} \oint_{\Gamma} f(z) \frac{s(z)^3}{1-s(z)^2} \dd z -\frac{1}{2\pi \ii} \oint_{\Gamma} f(z) \frac{\sqrt{\lambda\fh} s'(z)}{1+\sqrt{\lambda\fh} s(z)} \dd z \\
	&\qquad -\frac{\sqrt{\lambda} (\sqrt{\fh_d} - \sqrt{\fh})}{2\pi \ii} \oint_{\Gamma} f(z) s'(z) \dd z + o(1)
\end{split} \eeq
with high probability.
From \eqref{eq:tM_chain_final}, we thus find that the CLT for the LSS holds, i.e.,
\beq \label{eq:CLT_entry}
	\left( \sum_{i=1}^N f(\mu^{\tM}_i) - N \int_{-2}^2 \frac{\sqrt{4-x^2}}{2\pi} f(x) \, \dd x \right) \to \caN(m_{\tM}(f), V_{\tM}(f)),
\eeq
and the variance $V_{\tM}(f) = V_W(f)$ since the second term in \eqref{eq:tM_chain_final} converges to a deterministic number as $N \to \infty$, which corresponds to the change of the mean. In particular,
\beq \begin{split} \label{eq:mean_change}
	m_{\tM}(f) - m_W(f) &= -\frac{\lambda (\gh - \fh)}{2\pi \ii} \oint_{\Gamma} f(z) s'(z) s(z) \dd z +\frac{1}{2\pi \ii} \oint_{\Gamma} f(z) \frac{\sqrt{\lambda\fh} s'(z)}{1+\sqrt{\lambda\fh} s(z)} \dd z \\
	&\qquad +\frac{\sqrt{\lambda} (\sqrt{\fh_d} - \sqrt{\fh})}{2\pi \ii} \oint_{\Gamma} f(z) s'(z) \dd z \\
	&= \frac{1}{2\pi \ii} \oint_{\Gamma} f(z) s'(z) \left[ -\lambda (\gh - \fh) s(z) + \frac{\sqrt{\lambda\fh}}{1+\sqrt{\lambda\fh} s(z)} + \sqrt{\lambda} (\sqrt{\fh_d} - \sqrt{\fh}) \right] \dd z.
\end{split} \eeq
Following the computation in the proof of Lemma 4.4 in \cite{Baik-Lee2017} with the identity $s'(z) = \frac{s(z)^2}{1-s(z)^2}$, we find that the right-hand side of \eqref{eq:mean_change} is given by
\beq \begin{split}
	&\frac{1}{2\pi \ii} \oint_{\Gamma} f(z) s'(z) \left[ -\lambda (\gh - \fh) s(z) + \frac{\sqrt{\lambda\fh}}{1+\sqrt{\lambda\fh} s(z)} + \sqrt{\lambda} (\sqrt{\fh_d} - \sqrt{\fh}) \right] \dd z \\
	&= (\sqrt{\lambda \fh_d} - \sqrt{\lambda \fh}) \tau_1 (f) + (\lambda\gh - \lambda\fh) \tau_2 (f) + \sum_{\ell=1}^{\infty} \sqrt{(\lambda\fh)^{\ell}} \tau_{\ell}(f).
\end{split} \eeq
(See also Remark 1.7 of \cite{Baik-Lee2017}.) This proves Theorem \ref{thm:trans_CLT}.


\subsection{Proof of Lemma \ref{lem:Tr_R^W}}

In this subsection, we prove Lemma \ref{lem:Tr_R^W}.

\subsubsection*{Notational remarks}

In the rest of the section, we use $C$ order to denote a constant that is independent of $N$. Even if the constant is different
from one place to another, we may use the same notation $C$ as long as it
does not depend on $N$ for the convenience of the presentation.

\begin{proof}[Proof of Lemma \ref{lem:Tr_R^W}]
To prove the lemma, we consider
\beq \begin{split} \label{eq:derivative_R^W}
	\frac{\partial}{\partial \theta} \Tr R^W(\theta, z) &= -\sum_{i=1}^N \sum_{a, b=1}^N \frac{\partial W_{ab}(\theta)}{\partial \theta} R^W_{ia}(\theta, z) R^W_{bi}(\theta, z) \\
	&= -\frac{\partial}{\partial z} \sum_{a, b=1}^N \frac{\partial W_{ab}(\theta)}{\partial \theta} R^W_{ba}(\theta, z),
\end{split} \eeq
where we again used that $\frac{\partial}{\partial z} R^W(\theta, z) = R^W(\theta, z)^2$. We expand the right-hand side by using the definition of $W(\theta)$,
\beq
	W_{ab}(\theta) = \left( 1-\theta + \theta \sqrt{N S_{ab}} \right) W_{ab},
\eeq
and get
\beq \begin{split} \label{eq:W_inter_1}
	\sum_{a, b=1}^N \frac{\partial W_{ab}(\theta)}{\partial \theta} R^W_{ba}(\theta, z) &= \sum_{a, b=1}^N \left(-1 + \sqrt{N S_{ab}} \right) W_{ab} R^W_{ba}(\theta, z) = \sum_{a, b=1}^N \frac{-1+\sqrt{N S_{ab}}}{1-\theta + \theta \sqrt{N S_{ab}}} W_{ab}(\theta) R^W_{ba}(\theta, z) \\
	&= \frac{N \lambda (\gh - \fh)}{2} \sum_{a, b=1}^N x_a^2 x_b^2 W_{ab}(\theta) R^W_{ba}(\theta, z) + \caO( \sqrt{N} \| \bsx \|_{\infty}^{2}).
\end{split} \eeq
Here, we used the properties that $W_{ab}(\theta) = \caO(N^{-\frac{1}{2}})$, $R^W_{ba}(\theta, z)= \caO(N^{-\frac{1}{2}})$ for $b \neq a$, $R^W_{aa}(\theta, z) = \caO(1)$, and $\sum_a x_a^2 = \sum_b x_b^2 = 1$, which imply
\beq
	\left| N^2 \sum_{a, b=1}^N x_a^4 x_b^4 W_{ab}(\theta) R^W_{ba}(\theta, z) \right| \leq N^2 \| \bsx \|_{\infty}^{4} \sum_{a, b=1}^N x_a^2 x_b^2 |W_{ab}(\theta) R^W_{ba}(\theta, z)| = \caO( N \| \bsx \|_{\infty}^{4})
\eeq
and
\beq
	\left| N \sum_{a=1}^N x_a^4 W_{aa}(\theta) R^W_{aa}(\theta, z) \right| \leq N \| \bsx \|_{\infty}^{2} \sum_{a=1}^N x_a^2 |W_{aa}(\theta) R^W_{aa}(\theta, z)| = \caO( \sqrt{N} \| \bsx \|_{\infty}^{2}).
\eeq

Since $W(\theta) R^W(\theta, z) = I + zR^W(\theta, z)$,
\beq \begin{split}
	\sum_{a, b=1}^N x_b^2 W_{ab}(\theta) R^W_{ba}(\theta, z) &= \sum_{b=1}^N x_b^2 (W (\theta) R^W (\theta, z))_{bb} = 1 + z \sum_{b=1}^N x_b^2 R^W_{bb} (\theta, z) \\
	&= 1 + zs(z) + \caO(N^{-\frac{1}{2}}).
\end{split} \eeq
Plugging it into \eqref{eq:W_inter_1}, we get
\beq \begin{split} \label{eq:W_inter_2}
	&\sum_{a, b=1}^N \frac{\partial W_{ab}(\theta)}{\partial \theta} R^W_{ba}(\theta, z) \\
	&= \frac{\lambda (\gh - \fh)}{2} (1+zs(z)) + \frac{N\lambda (\gh - \fh)}{2} \sum_{a, b=1}^N \big( x_a^2 - \frac{1}{N} \big) x_b^2 W_{ab}(\theta) R^W_{ba}(\theta, z) + \caO(\sqrt{N} \| \bsx \|_{\infty}^{2}).
\end{split} \eeq

It remains to estimate the second term in the right-hand side of \eqref{eq:W_inter_2}. Set
\beq \label{eq:X_def}
	X \equiv X(\theta, z) := \sum_{a, b=1}^N \big( x_a^2 - \frac{1}{N} \big) x_b^2 W_{ab}(\theta) R^W_{ba}(\theta, z).
\eeq
We notice that $|X| = \caO(N^{-1})$ on $\Gamma_{1/2}$ by a naive power counting as in \eqref{eq:W_inter_1}.
To obtain a better bound for $X$, we use a method based on a recursive moment estimate, introduced in \cite{LeeSchnelli2018}. We need the following lemma:

\begin{lem} \label{lem:recursive}
Let $X$ be as in \eqref{eq:X_def}. Define an event $\Omega_{\epsilon}$ by
\[
	\Omega_{\epsilon} = \bigcap_{i, j=1}^N \left( \{ |W_{ij}(\theta)| \leq N^{-\frac{1}{2}+\epsilon} \} \cap \{ |R^W_{ij}(\theta, z) - \delta_{ij} s(z)| \leq N^{-\frac{1}{2}+\epsilon} \} \right).
\]
Then, for any fixed (large) $D$ and (small) $\epsilon$, which may depend on $D$, 
\beq \begin{split} \label{eq:recursive}
	\E[|X|^{2D} | \Omega_{\epsilon}] &\leq C N^{\frac{1}{2}+\epsilon} \| \bsx \|_{\infty}^{4} \E[|X|^{2D-1} | \Omega_{\epsilon}] + C N^{1+4\epsilon} \| \bsx \|_{\infty}^{8} \E[|X|^{2D-2} | \Omega_{\epsilon}] \\
	&\quad + C N^{1+5\epsilon} \| \bsx \|_{\infty}^{12} \E[|X|^{2D-3} | \Omega_{\epsilon}] + N^{1+9\epsilon} \| \bsx \|_{\infty}^{16} \E[|X|^{2D-4} | \Omega_{\epsilon}].
\end{split} \eeq
\end{lem}

We will prove Lemma \ref{lem:recursive} at the end of this section. With Lemma \ref{lem:recursive}, we are ready to obtain an improved bound for $X$. First, note that $\p(\Omega_{\epsilon}^c) < N^{-D^2}$, which can be checked by applying a high-order Markov inequality with the moment condition on $\tM$ (Assumption \ref{assump:entry}(iii)). We decompose $\E[|X|^{2D}]$ by
\beq \begin{split} \label{eq:X_decompose}
	\E[|X|^{2D}] = \E[|X|^{2D} \cdot \mathbf{1}(\Omega_{\epsilon})] + \E[|X|^{2D} \cdot \mathbf{1}(\Omega_{\epsilon}^c)] = \E[|X|^{2D} | \Omega_{\epsilon}] \cdot \p(\Omega_{\epsilon}) + \E[|X|^{2D} \cdot \mathbf{1}(\Omega_{\epsilon}^c)].
\end{split} \eeq
The second term in the right-hand side of \eqref{eq:X_decompose}, the contribution from the exceptional event $\Omega_{\epsilon}^c$ is negligible, since $\p(\Omega_{\epsilon}^c) < N^{-D^2}$,
\beq \label{eq:Omega^c_1}
	\E[|X|^{2D} \cdot \mathbf{1}(\Omega_{\epsilon}^c)] \leq \left( \E[|X|^{4D} ] \right)^{\frac{1}{2}} \left( \p(\Omega_{\epsilon}^c) \right)^{\frac{1}{2}} \leq N^{-\frac{D^2}{2}} \left( \E[|X|^{4D} ] \right)^{\frac{1}{2}} 
\eeq
and
\beq \label{eq:Omega^c_2}
	\E[|X|^{4D}] \leq \left( \sum_{a, b=1}^N |W_{ab} R^W_{ba}| \right)^{4D} \leq \frac{N^{8D}}{(\im z)^{4D}} \max_{a, b} \E[|W_{ab}|^{4D}] \leq N^{10D},
\eeq
where we used a trivial bound $|R^W_{ba}| \leq \| R^W \| \leq \frac{1}{\im z}$. 

From Young's inequality
\[
	ab \leq \frac{a^p}{p} + \frac{b^q}{q},
\]
which holds for any $a, b > 0$ and $p, q > 0$ with $\frac{1}{p} + \frac{1}{q} = 1$, we find that
\beq \begin{split}
	N^{\frac{1}{2}+\epsilon} \| \bsx \|_{\infty}^{4} |X|^{2D-1} &= N^{\frac{(2D-1)\epsilon}{2D}} N^{\frac{1}{2}+\epsilon} \| \bsx \|_{\infty}^{4} \cdot N^{-\frac{(2D-1)\epsilon}{2D}} |X|^{2D-1} \\
	&\leq \frac{1}{2D} N^{(2D-1)\epsilon} (N^{\frac{1}{2}+\epsilon} \| \bsx \|_{\infty}^{4})^{2D} + \frac{2D-1}{2D} N^{-\epsilon} |X|^{2D}.
\end{split} \eeq
Applying Young's inequality for other terms in \eqref{eq:recursive}, we get
\beq \begin{split}
	\E[|X|^{2D} | \Omega_{\epsilon}] &\leq C N^{(2D-1)\epsilon} (N^{\frac{1}{2}+\epsilon} \| \bsx \|_{\infty}^{4})^{2D} + C N^{(D-1)\epsilon} (N^{1+4\epsilon} \| \bsx \|_{\infty}^{8})^D \\
	& \quad + C N^{(\frac{2D}{3}-1)\epsilon} (N^{1+5\epsilon} \| \bsx \|_{\infty}^{12})^{\frac{2D}{3}} + C N^{(\frac{D}{2}-1)\epsilon} (N^{1+9\epsilon} \| \bsx \|_{\infty}^{16})^{\frac{D}{2}} + C N^{-\epsilon} \E[|X|^{2D} | \Omega_{\epsilon}].
\end{split} \eeq
Absorbing the last term in the right-hand side to the left-hand side and plugging the estimates \eqref{eq:Omega^c_1} and \eqref{eq:Omega^c_2} into \eqref{eq:X_decompose}, we now get
\beq \begin{split}
		\E[|X|^{2D}] &\leq C N^{(2D-1)\epsilon} (N^{\frac{1}{2}+\epsilon} \| \bsx \|_{\infty}^{4})^{2D} + C N^{(D-1)\epsilon} (N^{1+4\epsilon} \| \bsx \|_{\infty}^{8})^D \\
	& \quad + C N^{(\frac{2D}{3}-1)\epsilon} (N^{1+5\epsilon} \| \bsx \|_{\infty}^{12})^{\frac{2D}{3}} + C N^{(\frac{D}{2}-1)\epsilon} (N^{1+9\epsilon} \| \bsx \|_{\infty}^{16})^{\frac{D}{2}} +  N^{-\frac{D^2}{2} + 5D}.
\end{split} \eeq
For any fixed $\epsilon' > 0$ independent of $D$, from the $(2D)$-th order Markov inequality,
\beq
	\p \big(|X| \geq N^{\epsilon'} \sqrt{N} \| \bsx \|_{\infty}^{4} \big) \leq N^{-2D\epsilon'} \frac{\E[|X|^{2D}]}{(\sqrt{N} \| \bsx \|_{\infty}^{4})^{2D}} \leq N^{-2D\epsilon'} N^{5D\epsilon}.
\eeq
Thus, by choosing $D$ sufficiently large and $\epsilon = 1/D$, we find that
\[
	|X| = \caO(\sqrt{N} \| \bsx \|_{\infty}^{4}).
\]

We now go back to \eqref{eq:derivative_R^W} and use \eqref{eq:W_inter_2} with the bound $|X| = \caO(\sqrt{N} \| \bsx \|_{\infty}^{4})$. Since $\| \bsx \|_{\infty} = O(N^{-\phi})$ for some $\frac{3}{8} < \phi \leq \frac{1}{2}$,
\beq
	\sum_{a, b=1}^N \frac{\partial W_{ab}(\theta)}{\partial \theta} R^W_{ba}(\theta, z) = \frac{\lambda (\gh - \fh)}{2} (1+zs(z)) + \caO(N^{\frac{3}{2}} \| \bsx \|_{\infty}^{4}).
\eeq
To handle the derivative of the right-hand side, we use Cauchy's integral formula as in \eqref{eq:z_deriavtiave2} with a rectangular contour, contained in $\Gamma^{\epsilon}$, whose perimeter is larger than $\epsilon$. Then, we get from \eqref{eq:derivative_R^W} that
\beq
	\frac{\partial}{\partial \theta} \Tr R^W(\theta, z) = -\frac{\lambda (\gh - \fh)}{2} \frac{\partial}{\partial z} (1+zs(z)) + \caO(N^{\frac{3}{2}} \| \bsx \|_{\infty}^{4}).
\eeq
Since $1+zs(z)+s(z)^2=0$,
\beq
	\frac{\partial}{\partial z} (1+zs(z)) = \frac{\partial}{\partial z} (-s(z)^2) = -2s(z) s'(z).
\eeq
After integrating over $\theta$ from $0$ to $1$, we conclude that \eqref{eq:W_inter_claim} holds for a fixed $z \in \Gamma_{1/2}$. To prove the uniform bound in the lemma, we can use the lattice argument in Section \ref{sec:CLT_proof}; see Equations \eqref{eq:fundamental}-\eqref{eq:lattice}.
\end{proof}

Finally, we prove the recursive moment estimate in Lemma \ref{lem:recursive}.

\begin{proof}[Proof of Lemma \ref{lem:recursive}]
We consider
\[
	\E[|X|^{2D}] = \E\left[ \sum_{a, b=1}^N \big( x_a^2 - \frac{1}{N} \big) x_b^2 W_{ab}(\theta) R^W_{ba}(\theta, z) X^{D-1} \overline{X}^{D} \right].
\]
For simplicity, we omit the $\theta$-dependence and $z$-dependence of $W \equiv W(\theta)$ and $R^W \equiv R^W(\theta, z)$.

We use the following inequality that generalizes Stein's lemma (see Proposition 5.2 of \cite{BaikLeeWu}):
Let $\Phi$ be a $C^2$ function. Fix a (small) $\epsilon > 0$, which may depend on $D$. Recall that $\Omega_{\epsilon}$ is the complement of the exceptional event on which $|W_{ab}|$ or $|R^W_{ba}|$ is exceptionally large for some $a, b$, defined by
\[
	\Omega_{\epsilon} = \bigcap_{i, j=1}^N \left( \{ |W_{ij}| \leq N^{-\frac{1}{2}+\epsilon} \} \cap \{ |R^W_{ij} - \delta_{ij} s| \leq N^{-\frac{1}{2}+\epsilon} \} \right).
\]
Then,
\beq \label{eq:Stein}
	\E[W_{ab} \Phi(W_{ab}) | \Omega_{\epsilon}] = \E[W_{ab}^2] \E[\Phi'(W_{ab}) | \Omega_{\epsilon}] + \epsilon_1,
\eeq
where the error term $\epsilon_1$ admits the bound
\beq \label{eq:Stein_error}
	|\epsilon_1| \leq C_1 \E \Big[ |W_{ab}|^3 \sup_{|t| \leq 1} \Phi''(t W_{ab}) \Big| \Omega_{\epsilon} \Big]
\eeq
for some constant $C_1$. The estimate \eqref{eq:Stein} follows from the proof of Proposition 5.2 of \cite{BaikLeeWu} with $p=1$, where we use the inequality (5.38) therein only up to second to the last line.

In the estimate \eqref{eq:Stein}, we let
\beq
	\Phi(W_{ab}) = R^W_{ba} X^{D-1} \overline{X}^{D}
\eeq
so that
\beq \begin{split} \label{eq:X_expansion}
	\E[|X|^{2D} | \Omega_{\epsilon}] = \sum_{a, b=1}^N \big( x_a^2 - \frac{1}{N} \big) x_b^2 \E \left[ W_{ab} \Phi(W_{ab}) | \Omega_{\epsilon} \right].
\end{split} \eeq

We now consider the term $\E \left[ W_{ab} \Phi(W_{ab}) | \Omega_{\epsilon} \right]$ in \eqref{eq:X_expansion}. Applying the equation \eqref{eq:Stein},
\beq \begin{split} \label{eq:Phi_expansion}
	&\E[ W_{ab} \Phi(W_{ab}) | \Omega_{\epsilon}] = \E[W_{ab}^2] \E[\Phi'(W_{ab}) | \Omega_{\epsilon}] + \epsilon_1 \\
	&= \E[W_{ab}^2] \left( -\E \left[R^W_{bb} R^W_{aa} X^{D-1} \overline{X}^{D} | \Omega_{\epsilon} \right] - \E \left[R^W_{ba} R^W_{ba} X^{D-1} \overline{X}^{D} | \Omega_{\epsilon} \right] \right. \\
	&\qquad \left. +(D-1) \E \left[ R^W_{ba} \frac{\partial X}{\partial W_{ab}} X^{D-2} \overline{X}^{D} \big| \Omega_{\epsilon} \right] + D \E \left[ R^W_{ba} \frac{\partial \overline{X}}{\partial W_{ab}} X^{D-1} \overline{X}^{D-1} \big| \Omega_{\epsilon} \right] \right) + \epsilon_1.
\end{split} \eeq
We plug it into \eqref{eq:X_expansion} and estimate each term. We decompose the term originated from the first term in \eqref{eq:Phi_expansion} as
\beq \begin{split}
	&\sum_{a, b=1}^N \big( x_a^2 - \frac{1}{N} \big) x_b^2 \E[W_{ab}^2] \E\left[R^W_{bb} R^W_{aa} X^{D-1} \overline{X}^{D} | \Omega_{\epsilon} \right] \\
	&= \sum_{a, b=1}^N \big( x_a^2 - \frac{1}{N} \big) x_b^2 \E[W_{ab}^2] \E\left[R^W_{bb} (R^W_{aa} - s) X^{D-1} \overline{X}^{D} | \Omega_{\epsilon} \right] + s \sum_{a, b=1}^N \big( x_a^2 - \frac{1}{N} \big) x_b^2 \E[W_{ab}^2] \E\left[R^W_{bb} X^{D-1} \overline{X}^{D} | \Omega_{\epsilon} \right].
\end{split} \eeq
The first term satisfies that
\beq \begin{split}
	&\left| \sum_{a, b=1}^N \big( x_a^2 - \frac{1}{N} \big) x_b^2 \E[W_{ab}^2] \E\left[R^W_{bb} (R^W_{aa} - s) X^{D-1} \overline{X}^{D} | \Omega_{\epsilon} \right] \right| \\
	&\leq C N^2 \| \bsx \|_{\infty}^{4} N^{-1} N^{-\frac{1}{2}+\epsilon} \E[|X|^{2D-1} | \Omega_{\epsilon}] = C N^{\frac{1}{2}+\epsilon} \| \bsx \|_{\infty}^{4} \E[|X|^{2D-1} | \Omega_{\epsilon}]
\end{split} \eeq
for some constant $C$. For the second term, we recall that $\sum_a (x_a^2 - \frac{1}{N}) = 0$ and $\E[W_{ab}^2]$ are identical except for $a \neq b$. Thus,
\beq \begin{split}
	&\left| s \sum_{a, b=1}^N \big( x_a^2 - \frac{1}{N} \big) x_b^2 \E[W_{ab}^2] \E\left[R^W_{bb} X^{D-1} \overline{X}^{D} | \Omega_{\epsilon} \right] \right| \\
	&\leq C \left| \sum_{b=1}^N \big| x_b^2 - \frac{1}{N} \big| x_b^2 |w_2-1| N^{-1} \E\left[R^W_{bb} X^{D-1} \overline{X}^{D} | \Omega_{\epsilon} \right] \right| \\
	&\leq C' N \| \bsx \|_{\infty}^{4} N^{-1} \E[|X|^{2D-1} | \Omega_{\epsilon}] = C' \| \bsx \|_{\infty}^{4} \E[|X|^{2D-1} | \Omega_{\epsilon}]
\end{split} \eeq
for some constants $C$ and $C'$. We then find that
\beq \label{eq:X_ex_1}
	\sum_{a, b=1}^N \big( x_a^2 - \frac{1}{N} \big) x_b^2 \E[W_{ab}^2] \E\left[R^W_{bb} R^W_{aa} X^{D-1} \overline{X}^{D} | \Omega_{\epsilon} \right] \leq C N^{\frac{1}{2}+\epsilon} \| \bsx \|_{\infty}^{4} \E[|X|^{2D-1} | \Omega_{\epsilon}]
\eeq
for some constant $C$. For the second term in \eqref{eq:Phi_expansion}, we also have
\beq \label{eq:X_ex_2}
	\sum_{a, b=1}^N \big( x_a^2 - \frac{1}{N} \big) x_b^2 \E[W_{ab}^2] \E\left[R^W_{ba} R^W_{ba} X^{D-1} \overline{X}^{D} | \Omega_{\epsilon} \right] \leq C N^{2\epsilon} \| \bsx \|_{\infty}^{4} \E[|X|^{2D-1} | \Omega_{\epsilon}].
\eeq
To estimate the third term and the fourth term in \eqref{eq:Phi_expansion}, we notice that on $\Omega_{\epsilon}$
\beq
	\left| \frac{\partial X}{\partial W_{ab}} \right| = \left| \sum_{i, j=1}^N \big( x_i^2 - \frac{1}{N} \big) x_j^2 W_{ij} R^W_{bi} R^W_{ja} + \big( x_a^2 - \frac{1}{N} \big) x_b^2 R^W_{ba} \right| \leq C N^{\frac{1}{2}+3\epsilon} \| \bsx \|_{\infty}^{4}.
\eeq
for some constant $C$. Thus, we obtain that
\beq \label{eq:X_ex_3}
	\sum_{a, b=1}^N \big( x_a^2 - \frac{1}{N} \big) x_b^2 \E[W_{ab}^2] \E \left[ R^W_{ba} \frac{\partial X}{\partial W_{ab}} X^{D-2} \overline{X}^{D} \big| \Omega_{\epsilon} \right] \leq C N^{1+4\epsilon} \| \bsx \|_{\infty}^{8} \E[|X|^{2D-2} | \Omega_{\epsilon}]
\eeq
and
\beq \label{eq:X_ex_4}
	\sum_{a, b=1}^N \big( x_a^2 - \frac{1}{N} \big) x_b^2 \E[W_{ab}^2] \E \left[ R^W_{ba} \frac{\partial \overline{X}}{\partial W_{ab}} X^{D-1} \overline{X}^{D-1} \big| \Omega_{\epsilon} \right] \leq C N^{1+4\epsilon} \| \bsx \|_{\infty}^{8} \E[|X|^{2D-2} | \Omega_{\epsilon}].
\eeq
Hence, from \eqref{eq:Phi_expansion}, \eqref{eq:X_ex_1}, \eqref{eq:X_ex_2}, \eqref{eq:X_ex_3}, and \eqref{eq:X_ex_4}, 
\beq \begin{split} \label{eq:X_main}
	\left| \sum_{a, b=1}^N \big( x_a^2 - \frac{1}{N} \big) x_b^2 \E \left[ W_{ab} \Phi(W_{ab}) | \Omega_{\epsilon} \right] \right| &\leq C N^{\frac{1}{2}+\epsilon} \| \bsx \|_{\infty}^{4} \E[|X|^{2D-1} | \Omega_{\epsilon}] \\
	&\quad + C N^{1+4\epsilon} \| \bsx \|_{\infty}^{8} \E[|X|^{2D-2} | \Omega_{\epsilon}] + \epsilon_1.
\end{split} \eeq

It remains to estimate $|\epsilon_1|$ in \eqref{eq:Stein_error}. Proceeding as before,
\beq \begin{split} \label{eq:epsilon_1_1}
	&\sum_{a, b=1}^N \big( x_a^2 - \frac{1}{N} \big) x_b^2 \E \Big[ |W_{ab}|^3\Phi''(W_{ab}) \Big| \Omega_{\epsilon} \Big] \\
	&\leq C N^{\epsilon} \| \bsx \|_{\infty}^{4} \E[|X|^{2D-1} | \Omega_{\epsilon}] + C N^{1+2\epsilon} \| \bsx \|_{\infty}^{8} \E[|X|^{2D-2} | \Omega_{\epsilon}] + C N^{1+5\epsilon} \| \bsx \|_{\infty}^{12} \E[|X|^{2D-3} | \Omega_{\epsilon}].
\end{split} \eeq
We want to compare $\Phi''(W_{ab})$ and $\Phi''(tW_{ab})$ for some $|t|<1$. Let $R^{W, t}$ be the resolvent of $W$ where $W_{ab}$ and $W_{ba}$ are replaced by $tW_{ab}$ and $tW_{ba}$, respectively, and let $X^t$ be defined as $X$ in \eqref{eq:X_def} with the same replacement for $W_{ab}$ (and $W_{ba}$) and also $R^W$ is replaced by $R^{W, t}$. Then,
\beq \label{eq:R^W,t-R^W}
	R^{W, t}_{ji} - R^W_{ji} = (1-t) R^W_{ja} W_{ab} R^{W, t}_{bi},
\eeq
and
\beq
	X^t - X = \sum_{i, j=1}^N \big( x_i^2 - \frac{1}{N} \big) x_j^2 W_{ij} (R^{W, t}_{ji} - R^W_{ji}) -(1-t) \big( x_a^2 - \frac{1}{N} \big) x_b^2 W_{ab} R^{W, t}_{ba}.
\eeq
Thus, on $\Omega_{\epsilon}$,
\beq \label{eq:X_t-X}
	|X^t - X| \leq C N^{4\epsilon} \| \bsx \|_{\infty}^{4}.
\eeq
Using the estimates \eqref{eq:R^W,t-R^W} and \eqref{eq:X_t-X}, on $\Omega_{\epsilon}$, we obtain that
\beq \label{eq:epsilon_1_2}
	|\Phi''(W_{ab}) - \Phi''(tW_{ab})| \leq C |\Phi''(W_{ab})| + N^{\frac{1}{2}+5\epsilon} \| \bsx \|_{\infty}^{12} |X|^{2D-4}
\eeq
uniformly on $t \in (-1, 1)$.

Combining \eqref{eq:X_expansion} and \eqref{eq:X_main} with \eqref{eq:epsilon_1_1}, \eqref{eq:epsilon_1_2}, and \eqref{eq:Stein_error}, we finally get
\beq \begin{split}
	\E[|X|^{2D} | \Omega_{\epsilon}] &\leq C N^{\frac{1}{2}+\epsilon} \| \bsx \|_{\infty}^{4} \E[|X|^{2D-1} | \Omega_{\epsilon}] + C N^{1+4\epsilon} \| \bsx \|_{\infty}^{8} \E[|X|^{2D-2} | \Omega_{\epsilon}] \\
	&\quad + C N^{1+5\epsilon} \| \bsx \|_{\infty}^{12} \E[|X|^{2D-3} | \Omega_{\epsilon}] + C N^{1+9\epsilon} \| \bsx \|_{\infty}^{16} \E[|X|^{2D-4} | \Omega_{\epsilon}].
\end{split} \eeq
This proves the desired lemma.
\end{proof}

\end{appendix}


\end{document}